\documentclass[aism]{imsart}

\RequirePackage[authoryear]{natbib}

\RequirePackage{amsthm,amsmath,amsfonts,amssymb}

\RequirePackage[colorlinks,citecolor=blue,urlcolor=blue]{hyperref}
\RequirePackage{graphicx}

\startlocaldefs

\theoremstyle{plain}

\newtheorem{theorem}{Theorem}
\newtheorem{lemma}[theorem]{Lemma}
\newtheorem{proposition}[theorem]{Proposition}
\newtheorem{corollary}[theorem]{Corollary}
\theoremstyle{remark}
\newtheorem{definition}{Definition}[section]
\newtheorem*{example}{Example}

\usepackage{amsmath,amssymb,amsfonts,amsthm,bbm,empheq}
\usepackage{epic,eepic,epsfig,longtable}
\usepackage{multirow,multicol,verbatim}
\usepackage{array}
\usepackage{lscape}
\usepackage{graphicx}
\usepackage{latexsym}
\usepackage{comment}
\usepackage{booktabs}
\usepackage{epsfig}
\usepackage{CJK}
\usepackage{color}
\usepackage{mathtools}
\usepackage{extarrows}
\usepackage{centernot}
\usepackage{caption}
\usepackage{subcaption}
\usepackage{tikz}
\usetikzlibrary{shapes.multipart,shapes.geometric,spy}

\NewCommandCopy{\baccent}{\b}
\AtBeginDocument{%
  \DeclareRobustCommand{\b}[1]{\ifmmode\mathbf{#1}\else\baccent{#1}\fi}%
}

\def\[{\left [}  \def\]{\right ]} \def\({\left (}  \def\){\right )}

\makeatletter
\def\underbar#1{\underline{\sbox\tw@{$#1$}\dp\tw@\z@\box\tw@}}
\makeatother

\def\newpage{\vfill\eject}
\def\today{\ifcase\month\or
  January\or February\or March\or April\or May\or June\or
  July\or August\or September\or October\or November\or December\fi
  \space\number\day, \number\year}

\usepackage{stackengine} 
\stackMath

\setstackgap{L}{8pt}

\allowdisplaybreaks
\usepackage{verbatim}

\usepackage{lipsum}
\usepackage{enumitem}
\usepackage{hyperref}
\usepackage{xcolor}

\def\:={\coloneqq}
\def\=:{\eqqcolon}
\theoremstyle{plain}
\ifx\theorem\undefined
  \newtheorem{theorem}{Theorem}[section]
\fi
\ifx\proposition\undefined
  \newtheorem{proposition}{Proposition}[section]
\fi
\ifx\corollary\undefined
  \newtheorem{corollary}{Corollary}[section]
\fi
\ifx\lemma\undefined
  \newtheorem{lemma}{Lemma}[section]
\fi

\theoremstyle{definition}
\ifx\definition\undefined
  \newtheorem{definition}{Definition}[section]
\fi
\ifx\assumption\undefined
  \newtheorem{assumption}{Assumption}
\fi
\ifx\example\undefined
  \newtheorem{example}{Example}[section]
\fi
\ifx\exercise\undefined
  \newtheorem{exercise}{Exercise}[section]
\fi
\ifx\conjecture\undefined
  \newtheorem{conjecture}{Conjecture}[section]
\fi

\theoremstyle{remark}
\ifx\remark\undefined
  \newtheorem{remark}{Remark}
\fi

\usepackage[capitalise,nameinlink,noabbrev]{cleveref}
\crefformat{equation}{Eq.\ (#2#1#3)}
\crefrangeformat{equation}{Eqs.\ (#3#1#4) to~(#5#2#6)}
\crefmultiformat{equation}{Eq.\ (#2#1#3)}%
{ and~(#2#1#3)}{, (#2#1#3)}{ and~(#2#1#3)}
\crefname{assumption}{Assumption}{Assumptions}
\crefname{theorem}{Theorem}{Theorems}
\crefname{proposition}{Proposition}{Propositions}
\crefname{corollary}{Corollary}{Corollaries}
\crefname{lemma}{Lemma}{Lemmas}
\crefname{remark}{Remark}{Remarks}
\crefname{definition}{Definition}{Definitions}
\crefname{example}{Example}{Examples}

\newcommand{\reels}{\mathbb{R}}
\newcommand{\naturels}{\mathbb{N}}

\newcommand{\esp}{\mathbb{E}}
\newcommand{\proba}{\mathbb{P}}
\newcommand{\Tau}{\mathrm{T}}

\endlocaldefs

\begin{document}

\begin{frontmatter}
\title{Non-explicit formula of boundary crossing probabilities by the Girsanov theorem}
\runtitle{Non-explicit formula of boundary crossing probabilities}
\thankstext{T1}{The online version of this article contains supplementary material.}

\begin{aug}

\author[A]{\fnms{Yoann}~\snm{Potiron*}\ead[label=e1]{potiron@fbc.keio.ac.jp}},
\address[A]{Faculty of Business and Commerce, Keio University. 2-15-45 Mita, Minato-ku, Tokyo, 108-8345, Japan\printead[presep={,\ }]{e1}}
\end{aug}

\begin{abstract}
This paper derives several formulae for the probability that a
Wiener process, which has a stochastic drift and random variance, crosses a one-sided stochastic boundary within a finite time
interval. A non-explicit formula is first obtained by the Girsanov theorem
when considering an equivalent probability measure in which the
boundary is constant and equal to its starting value. A more explicit formula is then achieved by decomposing the
Radon-Nikodym derivative inverse. This decomposition expresses
it as the product of a random variable, which is measurable with
respect to the Wiener process’s final value, and an independent
random variable. We also provide an explicit formula based on a strong theoretical assumption. To apply the Girsanov theorem, we assume that the difference
between the drift increment and the boundary increment, divided by the standard deviation, is absolutely
continuous. Additionally, we assume that its derivative satisfies
Novikov’s condition.
\end{abstract}



\end{frontmatter}

\section{Introduction}
 This paper concerns boundary crossing probabilities, i.e., the probability that a stochastic process  crosses a boundary. The application of boundary crossing probabilities in statistics dates back to the one-sample Kolmogorov-Smirnov statistic,
where the process represents the difference between
the true and empirical cumulative distribution functions (cdfs). The primary application of boundary crossing probabilities is in
sequential analysis. Initially, the focus was on boundary crossing probabilities for random walks. Due to the complexity of
solving this problem, the literature often relies on continuous
approximations and develops theoretical tools for cases where
the process is a Wiener process (see Gut (\citeyear{gut1974moments}), Woodroofe (\citeyear{woodroofe1976renewal}), Woodroofe (\citeyear{woodroofe1977second}), Lai et al. (\citeyear{lai1977nonlinear}), Lai et al. (\citeyear{lai1979nonlinear}) and Siegmund (\citeyear{siegmund1986boundary})). 

Another field of application is in survival analysis. Matthews et al. (\citeyear{matthews1985asymptotic}) show that tests for constant hazard involve the first-passage time (FPT) of an Ornstein-Uhlenbeck process. Butler et al. (\citeyear{butler1997stochastic}) present a Bayesian approach when the process is semi-Markovian. Eaton et al. (\citeyear{eaton1977length}) discuss the application of FPT for hospital stay. Aalen et al. (\citeyear{aalen2001understanding}) study the case when the process is Markovian. Detailed reviews on FPT are available in Lee et al. (\citeyear{lee2006threshold}) and Lawless (\citeyear{lawless2011statistical}) (Section 11.5, pp. 518-523).

Another application is in pricing barrier options in mathematical finance (see Roberts et al. (\citeyear{roberts1997pricing})). There are also some applications in econometrics. Abbring (\citeyear{abbring2012mixed}) studies mixed FPT of a spectrally negative Levy process. Renault et al. (\citeyear{renault2014dynamic}) considers mixed FPT of the sum of a Wiener process and a positive linear drift. Potiron et al. (\citeyear{potiron2017estimation}) estimate the quadratic covariation between two price processes based on endogenous observations generated by FPT of an It\^{o}-semimartingale to a stochastic boundary process.  

Despite their importance for applications, explicit formulae of these boundary crossing probabilities only exist when the boundaries and the drift are linear. More specifically, Doob (\citeyear{doob1949heuristic}) gives explicit formulae (Equations (4.2)-(4.3), pp. 397-398) based on elementary geometrical and analytical arguments. They are obtained when the final time is not finite, the variance is nonrandom, the drift is null and the boundaries are nonrandom linear with nonnegative upper trend and nonpositive lower trend. Malmquist (\citeyear{malmquist1954certain}) obtains an explicit formula conditioned on the starting and final values of the Wiener process for a finite final time (Theorem 1, p. 526). This is obtained with  Doob's transformation (Section 5, pp. 401-402) in the one-sided boundary case. Anderson (\citeyear{anderson1960modification}) derives an explicit formula conditioned on the final value of the Wiener process (Theorem 4.2, pp. 178-179) in the two-sided boundary case with linear drift. Then, he integrates it with respect to the final value of the Wiener process to get an explicit formula (Theorem 4.3, p. 180). 

For square root boundaries, Breiman (\citeyear{breiman1967first}) expresses the problem as boundary crossing probabilities of an Ornstein-Uhlenbeck process to a constant boundary. They are obtained with Doob's transformation. However, the boundary crossing probabilities of an Ornstein-Uhlenbeck process to a constant boundary are only known in the form of Laplace transform. Daniels (\citeyear{daniels1969minimum}) uses the same technique and obtains an explicit formula. Finally, the boundary crossing probabilities of a jump diffusion process with linear drift to a constant boundary are obtained in the form of Laplace transform in Kou et al. (\citeyear{kou2003first}). Alili et al. (\citeyear{alili2005some}), Doney et al. (\citeyear{doney2006overshoots}) and Kypriano et al. (\citeyear{kyprianou2010exact}) consider a link between the first and last passage time and overshoot above/below a fixed level of a Lévy process. Potiron (\citeyear{potiron2023explicit}) obtains an explicit formula when the boundary is constant and the stochastic process is a continuous local martingale.
 
Since there is no available explicit formula when the drift and the boundaries are not linear, there is a large literature on approximating and computing numerically these boundary crossing probabilities. Strassen (\citeyear{strassen1967almost}) (Lemma 3.3, p. 323) shows that $P_g^Z$ is continuously differentiable when $g$ is continuously differentiable. Durbin (\citeyear{durbin1971boundary}, Wang et al. (\citeyear{wang1997boundary}) and Novikov eet al. (\citeyear{novikov1999approximations}) use piecewise-linear boundaries to approximate the general boundaries. Durbin (\citeyear{durbin1985first}) gives
a formula for a general boundary, which depends on asymptotic conditional expectations whose approximations are studied in Salminen (\citeyear{salminen1988first}).

In this paper, we derive several formulae for the one-sided and two-sided boundary crossing probability when the boundaries and drift are stochastic processes and the variance is random. Unfortunately, these formulae are either non-explicit or explicit, but based on a strong theoretical assumption. We derive the results in two cases, i.e., (i) a simpler case when the one-sided boundary and the drift are nonrandom time-varying and the variance is nonrandom and (ii) a more complicated case when the one-sided boundary and the drift are stochastic processes and the variance is random.

More specifically, consider a stochastic process $(Z_t)_{t \in \reels^+}$ defined as $Z_t = \mu_t + \sigma W_t$.
Here, $(\mu_t)_{t \in \reels^+}$ is a stochastic drift process, $(W_t)_{t \in \reels^+}$ is a standard Wiener process with random
time-invariant variance $\sigma^2$, and $(g_t)_{t \in \reels^+}$ and $(h_t)_{t \in \reels^+}$ are two 
stochastic boundary processes.  We focus on the probabilities of a process crossing one-sided
and two-sided boundaries, defined as follows
\begin{eqnarray}
\label{def_bcp}
P_g^Z(T) & = & \proba \big( \sup_{0 \leq t \leq T} Z_t - g_t   \geq 0\big),\\ \label{def_bcp2} P_{g,h}^Z(T) & = & \proba \big( \sup_{0 \leq t \leq T} Z_t - g_t   \geq 0 \text{ or } \sup_{0 \leq t \leq T} h_t - Z_t  \geq 0 \big), 
\end{eqnarray}
i.e., the probability that the process $Z$ crosses the boundary or one of both boundaries between $0$ and the final time $T$.

We describe first the main results in the one-sided (i) case. A non-explicit formula is first obtained by the Girsanov theorem when considering an equivalent probability measure in which the boundary is constant equal to its starting value. To apply the Girsanov theorem, the main idea consists in rewriting the boundary crossing probability of a time-varying boundary as an equivalent boundary crossing probability of a constant boundary. More specifically, we define the new nonrandom drift as $u_t=\frac{\mu_t-g_t+g_0}{\sigma}$, the new $\sigma(W_t)$-measurable process as $Y_t = u_t + W_t$ and the new constant boundary as $b=\frac{g_0-\mu_0}{\sigma}$. We then observe that the boundary crossing probability (\ref{def_bcp}) may be rewritten as $P_g^Z(T)  = P_b^Y(T)$. We thus obtain (see Proposition \ref{th_condproba_nonrandom_oneboundary})
\begin{eqnarray}
\label{intro_eq}
\proba (\Tau^Y_b \leq T | W_T)  =  \esp_{\mathbb{Q}} \big[\mathbf{1}_{\{\Tau^Y_b \leq T \}} M_T^{-1} | W_T \big]. 
\end{eqnarray}
Here, $M_T=\exp{\big(\overline{W}_T - \frac{1}{2} \int_0^T \theta_s^2 ds\big)}$ is the Radon-Nikodym derivative with $\overline{W}_T = \int_0^T \theta_s dW_s$, $u_T=\int_0^T \theta_s ds$, and $\theta$ is a nonrandom function $\theta : [0,T] \rightarrow \reels$.

A more explicit formula is then obtained by using an elementary decomposition of the Radon-Nikodym derivative inverse. This decomposition expresses it as the product of a random variable, which is $\sigma(W_T)$-measurable, and an independent random variable. More specifically, the decomposition is based on the fact that $(W_T,\overline{W}_T)$ is a centered normal random vector under $\proba$. It consists of rewriting $\overline{W}_T$ as $\overline{W}_T = \alpha W_T + \widetilde{\alpha} \widetilde{W}$, where $\widetilde{W}$ is a standard normal random variable under $\proba$, independent of $W_T$. Then, we obtain (see Theorem \ref{th_approximation_nonrandom_oneboundary}) 
\begin{eqnarray}
\nonumber \proba (\Tau^Y_b \leq T | W_T)  = \\ \label{intro_eq1}  \exp{\big(-\alpha W_T + \frac{1}{2} \int_0^T \theta_s^2 ds\big)} \esp_{\mathbb{Q}} \big[\mathbf{1}_{\{\Tau^Y_b \leq T \}}  \exp{\big(-\widetilde{\alpha} \widetilde{W}\big)} | W_T \big].
\end{eqnarray}

We also give a formula based on a strong theoretical assumption. More specifically, we 
assume that 
$$\esp_{\mathbb{Q}} \big[\mathbf{1}_{\{\Tau^Y_b \leq T \}}  \exp{\big(-\widetilde{\alpha} \widetilde{W}\big)} | W_T \big] = \esp_{\mathbb{Q}} \big[\mathbf{1}_{\{\Tau^Y_b \leq T \}}  | W_T \big] \esp_{\mathbb{Q}} \big[  \exp{\big(-\widetilde{\alpha} \widetilde{W}\big)} | W_T \big].$$ Then, the non-explicit factor in Equation (\ref{intro_eq1}) is equal to (see Theorem \ref{th_approximation2_nonrandom_oneboundary}) 
\begin{eqnarray*}
\esp_{\mathbb{Q}} \big[\mathbf{1}_{\{\Tau^Y_b \leq T \}}  \exp{\big(-\widetilde{\alpha} \widetilde{W}\big)} | W_T \big] & = & \exp{\big(-\frac{2b(b-Y_T)}{T}\big)} \\ &&\times\exp{\big(\widetilde{\alpha} \int_0^T \widetilde{\theta}_s \theta_s ds\big)} \mathcal{L}_N(\widetilde{\alpha}).
\end{eqnarray*}
Here, $\mathcal{L}_N$ is the Laplace transform of a standard normal variable.  
This is based on the preliminary result (see Lemma \ref{prop_condproba_nonrandom_oneboundary1}) $$\mathbb{Q} (\Tau^Y_b \leq T | W_T)  =   \exp{\big(-\frac{2b(b-Y_T)}{T}\big)},$$ 
by using the explicit formula from Malmquist (\citeyear{malmquist1954certain}) (Theorem 1, p. 526) and since $Y$ is a standard Wiener process under $\mathbb{Q}$. 

 Finally, we give formulae of $P_b^Y(T)$ by integrating $\proba (\Tau^Y_b \leq T | W_T)$ with respect to the value of $W_T$ (see Corollary \ref{th_proba_nonrandom_oneboundary1}). We also derive similar formulae in the case (ii). If we define $v$ as $v=(g,\mu,\sigma)$ and we assume that $v$ is independent of $W$, the elementary idea in the case (ii) is to condition by both $W_T$ and $v$, i.e., to derive results of the form $\proba (\Tau^Y_b \leq T | W_T,v)$. We also derive similar formulae in the two-sided boundary case when the difference between the deviation of each boundary from their starting value is linear, i.e., there exists $\beta \in \reels$ such that $h_t - h_0 = g_t - g_0 + \beta t$. To apply the Girsanov theorem to the two-sided boundary case, we cannot use two different drifts since the process $Z$ is unique in the boundary crossing probability (\ref{def_bcp2}). Thus, the elementary idea consists in rewriting the FPT to a two-sided time-varying boundary as an equivalent FPT to a two-sided boundary, with one constant boundary and one linear boundary. More specifically, we define the new drift as $u_t=\frac{\mu_t-\mu_0-g_t+g_0}{\sigma}$, the new process as $Y_t = u_t + W_t$, the new constant boundary as $b=\frac{\mu_0-g_0}{\sigma}$ and the new linear boundary as $c_t=\frac{h_0-\mu_0+\beta t}{\sigma}$. We then observe that the boundary crossing probability (\ref{def_bcp2}) may be rewritten as $P_{g,h}^Z = P_{b,c}^Y$.

When the boundaries and the drift are linear, the one-sided and two-sided boundary crossing probability (\ref{def_bcp})-(\ref{def_bcp2}) can be obtained easily. For instance, one can use a combination of the Girsanov theorem and the reflection principle, or by calculating the Laplace
 transforms via some appropriate martingales and the optional sampling theorem. Details of both methods can be found in many classical textbooks on stochastic analysis, e.g., Karatzas et al. (\citeyear{karatzas2012brownian}) or Revuz et al. (\citeyear{revuz2013continuous}). More specifically, one can obtain by the Girsanov theorem Equation (\ref{intro_eq}), and since the Radon-Nikodym derivative inverse $M_T^{-1}$ is $\sigma(W_T)$-measurable in that simpler case, one can conclude by the joint distribution of the maximum and terminal value of a Wiener process based on the reflection principle. When the boundaries and the drift are not linear, however, $M_T^{-1}$ is no longer $\sigma(W_T)$-measurable. This renders a direct calculation not possible since that would require to extend the arguments based on the reflection principle. As discussed above, we circumvent that difficulty by an elementary decomposition of the Radon-Nikodym derivative inverse. This decomposition expresses it as the product of a $\sigma(W_T)$-measurable random variable and an independent random variable. The price to pay is that we do not obtain explicit formulae without strong theoretical assumptions. 

As it stands, our obtained formulae are based on a strong theoretical assumption. Unfortunately, such assumption is hard to show when the final time is fixed. The main reason is that the boundary is not linear enough. However, we conjecture that this assumption is asymptotically met if we divide the final time interval $[0,T]$ into smaller time intervals with length that converges to 0 asymptotically. This is due to the fact that the boundary gets more and more linear locally. Thus, we can relatively safely apply our obtained formulae locally. By taking the sum of the local approximations, we can then approximate the boundary crossing probabilities (\ref{def_bcp})-(\ref{def_bcp2}).

\section{Main results}
\label{sec_mainresultsonesided}
\subsection{One-sided time-varying boundary case}
In this section, we consider the case when the one-sided boundary and the drift are nonrandom time-varying and the variance is nonrandom. 

We consider the complete stochastic basis $\mathcal{B} = (\Omega, \proba, \mathcal{F}, \mathbf{F})$, where $\mathcal{F}$ is a $\sigma$-field and $\mathbf{F} = (\mathcal{F}_t)_{t \in \reels^+}$ is a filtration. We define the set of continuous functions from $\reels^+$ to $\reels$ as $\mathcal{C}(\reels^+,\reels)$. We first give the definition of the set of boundary functions. We assume that the boundary is continuous since it is required in the assumptions for the Girsanov theorem.
\begin{definition}\label{defboundaryset}
We define the set of boundary functions as $\mathcal{G} = \mathcal{C}(\reels^+,\reels).$
\end{definition}
 We now give the definition of the FPT. We assume that the stochastic process is continuous since we consider a Wiener process with a continuous drift which is required in the assumptions for the Girsanov theorem.
\begin{definition}\label{defFPT}
We define the FPT of an $\mathbf{F}$-adapted continuous process $Z$ to a boundary
$g \in \mathcal{G}$ satisfying
$Z_0 \leq g_0$ as
\begin{eqnarray}
\label{TgZdef}
\Tau_g^Z = \inf \{t \in \reels^+ \text{ s.t. } Z_t \geq g_t\}.
\end{eqnarray}
\end{definition}
 We have that $Z$ is a continuous and 
$\mathbf{F}$-adapted stochastic process and $\inf \{t \in \reels^+ \text{ s.t. } Z_t \geq g_t\} = \inf \{t \in \reels^+ \text{ s.t. } (t,Z_t) \in G\}$, where $G= \{(t,u) \in \reels^+ \times \reels \text{ s.t. } u \geq g_t\}$ is a closed subset of $\reels^2$. Thus, the FPT $\Tau_{g}^{Z}$ is an $\mathbf{F}$-stopping time by Theorem I.1.27 (p. 7) in Jacod et al. (\citeyear{JacodLimit2003}). We can rewrite the boundary crossing probability $P_g^Z$ as the cdf of $\Tau_g^Z$, i.e., 
\begin{eqnarray}
\label{PgZdef}
P_g^Z(t)= \proba (\Tau^Z_g \leq t) \text{ for any } t \geq 0.
\end{eqnarray}
We assume that $W$ is an $\mathbf{F}$-standard Wiener process. We assume that $Z_t = \mu_t + \sigma W_t$. Here, $\mu$ is time-varying, nonrandom, and satisfies $\mu_0 < g_0$. The variance $\sigma^2$ is time-invariant, nonrandom, and positive, i.e., $\sigma^2 > 0$. To apply the Girsanov theorem, the main elementary idea is to rewrite the FPT to a time-varying boundary as an equivalent FPT to a constant boundary. More specifically, we define the new nonrandom drift as $u_t=\frac{\mu_t-\mu_0-g_t+g_0}{\sigma}$, the new process as $Y_t = u_t + W_t$ and the new constant boundary as $b=\frac{g_0-\mu_0}{\sigma}$. We then observe that the FPT (\ref{TgZdef}) may be rewritten as 
\begin{eqnarray}
\Tau_g^Z = \Tau_{b}^Y.
\end{eqnarray}
Then, we will consider an equivalent probability measure under which the new process $Y$ will be a standard Wiener process. Accordingly, we provide the assumption which corresponds to Novikov's condition (see Novikov (\citeyear{novikov1972identity})) which is required to apply the Girsanov theorem (see Girsanov (\citeyear{girsanov1960transforming})). The proofs of this paper would hold with no change with the more general conditions obtained by Kazamaki (\citeyear{kazamaki1977problem}).
\begin{proof}[\textbf{Assumption A}]
We assume that $u \neq 0$, i.e., that there exists a $t \in [0,T]$ such that $u_t \neq 0$. We also assume that $u$ is absolutely continuous on $[0,T]$, i.e., there exists a nonrandom function $\theta : [0,T] \rightarrow \reels$ with $u_t=\int_0^t \theta_s ds$. Finally, we assume that $\int_0^T \theta_s^2 ds < \infty$.
\phantom\qedhere
\end{proof}
\begin{definition}
We define $M$ as
\begin{eqnarray}
\label{defM}
M_t=\exp{\big(\int_0^t \theta_s dW_s - \frac{1}{2} \int_0^t \theta_s^2 ds\big)} \text{ for any } 0 \leq t \leq T.
\end{eqnarray}
\end{definition}
 By \textbf{Assumption A}, $M$ satisfies Novikov's condition and thus is a positive martingale. We embed this result and its implications on an equivalent probability measure $\mathbb{Q}$ by the Girsanov theorem in the following lemma.
\begin{lemma} \label{lemma_girsanov} Under \textbf{Assumption A}, we have that $M$ is a positive martingale. Thus, we can consider an equivalent probability measure $\mathbb{Q}$ such that the Radon-Nikodym derivative is defined as $\frac{d \mathbb{Q}}{d \proba}=M_T$. Finally, $Y$ is a standard Wiener process under $\mathbb{Q}$. \end{lemma}
 Consequently, we obtain that $\esp_\proba \big[X \big]  =  \esp_{\mathbb{Q}} \big[X M_T^{-1}\big]$ for any $\mathcal{F}_T$-measurable random variable $X$ by a change of probability in the expectation. The next proposition reexpresses $\proba (\Tau^Y_b \leq T | W_T)$ under $\mathbb{Q}$ and is a result of interest in its own right although the obtained formulae are not explicit. The proof is based on Lemma \ref{lemma_girsanov} and its consequence in the particular case $X= \mathbf{1}_{\{\Tau^Y_b \leq T \}} \mathbf{1}_{E_T}$ where $E_T$ is a $\sigma(W_T)$-measurable event. We define $\overline{W}_t$ as
\begin{eqnarray}
\label{overlineWdef}
\overline{W}_t = \int_0^t \theta_s dW_s.
\end{eqnarray}
\begin{proposition}
\label{th_condproba_nonrandom_oneboundary}
Under \textbf{Assumption A}, we have
\begin{eqnarray}
\label{eq_condproba_nonrandom_oneboundary0}
\proba (\Tau^Y_b \leq T | W_T) & = & \esp_{\mathbb{Q}} \big[\mathbf{1}_{\{\Tau^Y_b \leq T \}} M_T^{-1} | W_T \big].
\end{eqnarray}
This can be reexpressed as
\begin{eqnarray}
\label{eq_condproba_nonrandom_oneboundary}
\proba (\Tau^Y_b \leq T | W_T) & = & \esp_{\mathbb{Q}} \big[ M_T^{-1} \esp_\mathbb{Q} \big[\mathbf{1}_{\{\Tau^Y_b \leq T \}} | W_T,\overline{W}_T \big] | W_T \big].
\end{eqnarray}
\end{proposition}
 It remains to calculate $\esp_{\mathbb{Q}} \big[\mathbf{1}_{\{\Tau^Y_b \leq T \}} M_T^{-1} | W_T \big]$ or $$\esp_{\mathbb{Q}} \big[ M_T^{-1} \esp_\mathbb{Q} \big[\mathbf{1}_{\{\Tau^Y_b \leq T \}} | W_T,\overline{W}_T \big] | W_T \big]$$ 
if we want to derive a completely explicit formula. Although $Y$ is a standard Wiener process under $\mathbb{Q}$ by Lemma \ref{lemma_girsanov}, the presence of $M_T^{-1}$ or $\overline{W}_T$ in the conditional expectation renders a direct calculation not possible. Indeed, we would need to extend the arguments based on the reflection principle. We circumvent that difficulty by an elementary decomposition of the Radon-Nikodym derivative inverse $M_T^{-1}$. More specifically, the decomposition is based on the fact that $(W_T,\overline{W}_T)$ is a centered normal random vector under $\proba$. It consists of rewriting $\overline{W}_T$ as $\overline{W}_T = \alpha W_T + \widetilde{\alpha} \widetilde{W}$, where $\widetilde{W}$ is a standard normal random variable under $\proba$, independent of $W_T$. We define the correlation under $\proba$ between $W_T$ and $\overline{W}_T$ as $\rho$, i.e., $\rho = \operatorname{Cor}_\proba(W_T, \overline{W}_T)$. 
\begin{lemma} \label{lemma_representation} Under \textbf{Assumption A}, we have that $\frac{\overline{W}_T}{\sqrt{\int_0^T \theta_s^2 ds}}$ is a standard normal random variable under $\proba$. We can also show that $\rho  = \frac{\int_0^T \theta_s ds}{T\int_0^T \theta_s^2 ds}$. Moreover, there exists a standard normal random variable $\widetilde{W}$ under $\proba$ which is independent of $W_T$, and such that $\overline{W}_T$ when normalized can be reexpressed as
\begin{eqnarray}
\label{def_Wtilde}
\frac{\overline{W}_T}{\sqrt{\int_0^T \theta_s^2 ds}} = \rho \frac{W_T}{\sqrt{T}} + \sqrt{1-\rho^2} \widetilde{W},
\end{eqnarray}
This can be reexpressed as
\begin{eqnarray}
\label{def_Wtilde0}
\overline{W}_T = \alpha W_T + \widetilde{\alpha} \widetilde{W},
\end{eqnarray}
where $\alpha = \rho\sqrt{T^{-1}\int_0^T \theta_s^2 ds}$ and $\widetilde{\alpha} = \sqrt{(1- \rho^2) \int_0^T \theta_s^2 ds}$. If we define $\widetilde{\theta}_t = \frac{\theta_s - \alpha}{\widetilde{\alpha}}$, we can reexpress $\widetilde{W}$ as
\begin{eqnarray}
\label{def_Wtilde1}
\widetilde{W} & = & \int_0^T \widetilde{\theta}_s dW_s.
\end{eqnarray}
Finally, $\widetilde{W}+\int_0^T \widetilde{\theta}_s \theta_s ds$ is a standard normal variable under $\mathbb{Q}$.
\end{lemma} 
 Our main result is the next theorem, which gives a more explicit formula to $\proba (\Tau^Y_b \leq T | W_T)$. The proof is based on Lemma \ref{lemma_representation}. 

\begin{theorem}
\label{th_approximation_nonrandom_oneboundary}
Under \textbf{Assumption A},
we have
\begin{eqnarray}
\label{eq_nonapproximation_nonrandom_oneboundary}
&&\proba (\Tau^Y_b \leq T | W_T) \\ \nonumber  & = & \exp{\big(-\alpha W_T + \frac{1}{2} \int_0^T \theta_s^2 ds\big)} \esp_{\mathbb{Q}} \big[\mathbf{1}_{\{\Tau^Y_b \leq T \}}  \exp{\big(-\widetilde{\alpha} \widetilde{W}\big)} | W_T \big].
\end{eqnarray}
\end{theorem}
 We first calculate $\mathbb{Q} (\Tau^Y_b \leq T | W_T)$, whose explicit formula is given in the following lemma. This reexpresses Malmquist (\citeyear{malmquist1954certain}) (Theorem 1, p. 526) under $\mathbb{Q}$, who considers the linear case $\mu_t=0$, $g_t= a t + b$ and $\sigma = 1$ under $\proba$. He obtains that
\begin{eqnarray}
& & \proba (\Tau^Z_g \leq T | W_T = x) \\ \nonumber &=& \exp{\Big(-\frac{2b(aT+b-x)}{T}\Big)}\mathbf{1}_{\{x \leq aT+b\}} + \mathbf{1}_{\{x > aT+b\}}
\end{eqnarray}
for any $x \in \reels$.
\begin{lemma}
\label{prop_condproba_nonrandom_oneboundary1}
Under \textbf{Assumption A}, we have
\begin{eqnarray}
\label{propeq_condproba_nonrandom_oneboundary1}
\mathbb{Q} (\Tau^Y_b \leq T | W_T) & = &  \exp{\Big(-\frac{2b(b-Y_T)}{T}\Big)}\mathbf{1}_{\{Y_T \leq b\}} + \mathbf{1}_{\{Y_T > b\}}.
\end{eqnarray}
\end{lemma}
 The next theorem gives an explicit formula based on a strong theoretical assumption (\ref{approximation_assumption}). The proof is based on Lemma \ref{prop_condproba_nonrandom_oneboundary1}. Let $N$ be a standard normal random variable under $\proba$. We define the Laplace transform of $N$ as 
\begin{eqnarray}
\label{def_laplace_transform}
\mathcal{L}_N(u) = \esp_{\proba} \big[  \exp{\big(-u N\big)} \big].
\end{eqnarray}

\begin{theorem}
\label{th_approximation2_nonrandom_oneboundary}
We assume that \textbf{Assumption A} and the following assumption
\begin{eqnarray}
\label{approximation_assumption}
\esp_{\mathbb{Q}} \big[\mathbf{1}_{\{\Tau^Y_b \leq T \}}  \exp{\big(-\widetilde{\alpha} \widetilde{W}\big)} | W_T \big] \\ \nonumber = \esp_{\mathbb{Q}} \big[\mathbf{1}_{\{\Tau^Y_b \leq T \}}  | W_T \big] \esp_{\mathbb{Q}} \big[  \exp{\big(-\widetilde{\alpha} \widetilde{W}\big)} | W_T \big]
\end{eqnarray}
holds. Then, we have
\begin{eqnarray}
\nonumber
\proba (\Tau^Y_b \leq T | W_T) & = & \exp{\big(-\alpha W_T + \frac{1}{2} \int_0^T \theta_s^2 ds\big)} \\ \nonumber & & \times \Big(\exp{\Big(-\frac{2b(b-Y_T)}{T}\Big)}\mathbf{1}_{\{Y_T \leq b\}} + \mathbf{1}_{\{Y_T > b\}}\Big) \\ \label{eq_approximation_nonrandom_oneboundary} & &  \times \exp{\big(\widetilde{\alpha} \int_0^T \widetilde{\theta}_s \theta_s ds\big)} \mathcal{L}_N(\widetilde{\alpha}).
\end{eqnarray}
\end{theorem}
Finally, we get $P_b^Y(T)$ in the next corollary by integrating $\proba (\Tau^Y_b \leq T | W_T)$ with respect to the value of $W_T$. The proof follows the steps of Equations (3) in Wang et al (\citeyear{wang1997boundary}) (p. 55). We define the standard Gaussian cdf as $\phi (t)=\int_{0}^{t} \frac{1}{\sqrt{2 \pi}} \exp{\Big(-\frac{u^2}{2}\Big)}  du$ for any $t \in \reels^+$. 
\begin{corollary}
\label{th_proba_nonrandom_oneboundary1}
Under \textbf{Assumption A}, we have
\begin{eqnarray}
\nonumber 
P_b^Y(T) & = & 1 - \phi \Big(\frac{b-u_T}{\sqrt{T}}\Big) + \int_{-\infty}^{b-u_T} \frac{1}{\sqrt{2 \pi T}} \exp{\Big(-\frac{x^2}{2T}\Big)} \\ & & \nonumber \times \exp{\big(-\alpha x + \frac{1}{2} \int_0^T \theta_s^2 ds\big)} \\ \label{eq_proba_nonrandom_oneboundary} &&\times\esp_{\mathbb{Q}} \big[\mathbf{1}_{\{\Tau^Y_b \leq T \}}  \exp{\big(-\widetilde{\alpha} \widetilde{W}\big)} | W_T =x \big] dx.
\end{eqnarray}
If we further assume (\ref{approximation_assumption}), we have
\begin{eqnarray}
\nonumber 
P_b^Y(T) & = & 1 - \phi \Big(\frac{b-u_T}{\sqrt{T}}\Big) \\ & & \nonumber+ \int_{-\infty}^{b-u_T} \frac{1}{\sqrt{2 \pi T}} \exp{\Big(-\frac{x^2}{2T}\Big)} \exp{\big(-\alpha x + \frac{1}{2} \int_0^T \theta_s^2 ds\big)}  \\ & & \label{eq_proba_approximation_nonrandom_oneboundary} \times \exp{\Big(-\frac{2b(b-u_T-x)}{T}\Big)}\exp{\big(\widetilde{\alpha} \int_0^T \widetilde{\theta}_s \theta_s ds\big)} \mathcal{L}_N(\widetilde{\alpha}) dx.
\end{eqnarray}
\end{corollary}
 We consider the particular case when the boundary is linear, there is no drift and the standard deviation is equal to unity $\mu_t=0$, $g_t= a t + b$ and $\sigma = 1$. Then, Corollary \ref{th_proba_nonrandom_oneboundary1} reduces to Wang et al. (\citeyear{wang1997boundary}) (Equation (2), p. 55), i.e., 
\begin{eqnarray*}
P_b^Y(T) & = & 1 - \phi (\frac{b+aT}{\sqrt{T}}) + \exp{\big(- 2 b a \big)} \phi (\frac{b- aT}{\sqrt{T}}).
\end{eqnarray*}

\subsection{One-sided stochastic boundary process case}
In this section, we consider the case when the one-sided boundary and the drift are stochastic processes and the variance is random.

We first give the definition of the set of stochastic boundary processes.
\begin{definition}\label{defboundarysetrandom}
We define the set of stochastic boundary processes as $\mathcal{H} = \reels^+ \times \Omega \rightarrow \reels$ such that for any $g \in \mathcal{H}$ and $\omega \in \Omega$ we have $g(\omega) \in \mathcal{G}$ and $g$ is $\mathbf{F}$-adapted.
\end{definition}
 We now give the definition of the FPT.
\begin{definition}\label{defFPTrandom}
We define the FPT of an $\mathbf{F}$-adapted continuous process $Z$ to a boundary 
$g \in \mathcal{H}$ satisfying $Z_0 \leq g_0$ $\forall \omega \in \Omega$ as
\begin{eqnarray}
\label{TgZdefrandom}
\Tau_g^Z = \inf \{t \in \reels^+ \text{ s.t. } Z_t \geq g_t\}.
\end{eqnarray}
\end{definition}
 We have that $Z-g$ is an $\mathbf{F}$-adapted continuous stochastic process and $\inf \{t \in \reels^+ \text{ s.t. }$ $Z_t \geq g_t\} = \inf \{t \in \reels^+ \text{ s.t. } Z_t - g_t \geq 0\} = \inf \{t \in \reels^+ \text{ s.t. } Z_t - g_t \in \reels^+\}.$ Thus, the FPT $\Tau_{g}^{Z}$ is an $\mathbf{F}$-stopping time by Theorem I.1.27 (p. 7) in Jacod et al. (\citeyear{JacodLimit2003}). We can rewrite the boundary crossing probability $P_g^Z$ as the cdf of $\Tau_g^Z$, i.e., 
\begin{eqnarray}
\label{PgZdef_random_oneboundary}
P_g^Z(t)= \proba (\Tau^Z_g \leq t) \text{ for any } t \geq 0.
\end{eqnarray}
We assume that $\mu$ is an $\mathbf{F}$-adapted stochastic process which satisfies $\proba(\mu_0 < g_0) = 1$. We also assume that the variance $\sigma^2$ is time-invariant, random, and such that $\proba(\sigma^2 = 0)=0$. Finally, we assume that $v$ is independent of $W$, where $v$ is defined as $v=(g,\mu,\sigma)$. 
\begin{proof}[\textbf{Assumption B}]
We assume that $\proba( \exists t \in [0,T] \text{ s.t } u_t \neq 0) = 1$. We also assume that $u$ is absolutely continuous on $[0,T]$, i.e., there exists a stochastic process $\theta : [0,T] \times \Omega \rightarrow \reels$ with $u_t=\int_0^t \theta_s ds$, a.s.. Finally, we assume that $\esp [\exp{\big(\frac{1}{2} \int_0^T \theta_s^2 ds\big)}] < \infty$.
\phantom\qedhere
\end{proof}
\begin{definition}
We define $M$ as
\begin{eqnarray}
\label{defM_random_oneboundary}
M_t=\exp{\big(\int_0^t \theta_s dW_s - \frac{1}{2} \int_0^t \theta_s^2 ds\big)} \text{ for any } 0 \leq t \leq T.
\end{eqnarray}
\end{definition}
 By \textbf{Assumption B}, $M$ satisfies Novikov's condition and thus is a positive martingale.
 \begin{lemma} \label{lemma_girsanov_random_oneboundary} Under \textbf{Assumption B}, we have that $M$ is a positive martingale. Thus, we can consider an equivalent probability measure $\mathbb{Q}$ such that the Radon-Nikodym derivative is defined as $\frac{d \mathbb{Q}}{d \proba}=M_T$. Finally, $Y$ is a standard Wiener process under $\mathbb{Q}$. \end{lemma}
 The elementary idea in this section is to condition by both $W_T$ and $v$, i.e., to derive results of the form $\proba (\Tau^Y_b \leq T | W_T,v)$. The next proposition reexpresses $\proba (\Tau^Y_b \leq T | W_T,v)$ under $\mathbb{Q}$. We define $\overline{W}_t$ as
\begin{eqnarray}
\label{overlineWdef_random_oneboundary}
\overline{W}_t = \int_0^t \theta_s dW_s.
\end{eqnarray}
\begin{proposition}
\label{th_condproba_random_oneboundary}
Under \textbf{Assumption B}, we have
\begin{eqnarray}
\label{eq_condproba_random_oneboundary0}
\proba (\Tau^Y_b \leq T | W_T,v) & = & \esp_{\mathbb{Q}} \big[\mathbf{1}_{\{\Tau^Y_b \leq T \}} M_T^{-1} | W_T ,v\big].
\end{eqnarray}
This can be reexpressed as
\begin{eqnarray}
\label{eq_condproba_random_oneboundary}
\proba (\Tau^Y_b \leq T | W_T,v) & = & \esp_{\mathbb{Q}} \big[ M_T^{-1} \esp_\mathbb{Q} \big[\mathbf{1}_{\{\Tau^Y_b \leq T \}} | W_T,\overline{W}_T,v \big] | W_T,v \big].
\end{eqnarray}
\end{proposition}
 To obtain a more explicit formula, we use an elementary decomposition of the Radon-Nikodym derivative inverse $M_T^{-1}$. This decomposition expresses
it as the product of a random variable which is $\sigma(W_T,v)$-measurable and a random variable conditionally independent from $W_T$ given $v$. Since $\theta$ is a stochastic process, we do not have that $\overline{W}_T$ is a normal random variable, but rather that it is a mixed normal random variable. The elementary idea is to normalize $\overline{W}_T$ by $\sqrt{\int_0^T \theta_s^2 ds}$, so that $(W_T,\frac{\overline{W}_T}{\sqrt{\int_0^T \theta_s^2 ds}})$ 
is a centered normal random vector under $\proba$. We define the correlation under $\proba$ between $W_T$ and $\frac{\overline{W}_T}{\sqrt{\int_0^T \theta_s^2 ds}}$ as $\rho$, i.e., $\rho = \operatorname{Cor}_\proba(W_T, \frac{\overline{W}_T}{\sqrt{\int_0^T \theta_s^2 ds}})$. 
\begin{lemma} \label{lemma_representation_random_oneboundary} Under \textbf{Assumption B}, we have that $\frac{\overline{W}_T}{\sqrt{\int_0^T \theta_s^2 ds}}$ is a standard normal random variable under $\proba$. We can also show that $\rho  = \frac{1}{T}\esp_\proba \Big[\frac{\int_0^T \theta_s ds }{\sqrt{\int_0^T \theta_s^2 ds}}\Big]$. Moreover, there exists a standard normal random variable $\widetilde{W}$ under $\proba$, which is independent of $W_T$, and such that $\overline{W}_T$, when normalized, can be reexpressed a.s. as
\begin{eqnarray}
\label{def_Wtilde_random_oneboundary}
\frac{\overline{W}_T}{\sqrt{\int_0^T \theta_s^2 ds}} = \rho \frac{W_T}{\sqrt{T}} + \sqrt{1-\rho^2} \widetilde{W}.
\end{eqnarray}
This can be reexpressed a.s. as
\begin{eqnarray}
\label{def_Wtilde0_random_oneboundary}
\overline{W}_T = \alpha W_T + \widetilde{\alpha} \widetilde{W},
\end{eqnarray}
where $\alpha = \rho\sqrt{T^{-1}\int_0^T \theta_s^2 ds}$ a.s. and $\widetilde{\alpha} = \sqrt{(1- \rho^2) \int_0^T \theta_s^2 ds}$ a.s.. If we define $\widetilde{\theta}_t = \frac{\theta_s - \alpha}{\widetilde{\alpha}}$, we can reexpress $\widetilde{W}$ a.s. as
\begin{eqnarray}
\label{def_Wtilde1_random_oneboundary}
\widetilde{W} & = & \int_0^T \widetilde{\theta}_s dW_s.
\end{eqnarray}
Moreover, $\widetilde{W}+\int_0^T \widetilde{\theta}_s \theta_s ds$ is a standard normal variable under $\mathbb{Q}$. Finally, the conditional distribution of $\widetilde{W}+\int_0^T \widetilde{\theta}_s \theta_s ds$ given $v$, i.e., $\mathcal{D}(\widetilde{W}+\int_0^T \widetilde{\theta}_s \theta_s ds | v)$, is standard normal under $\mathbb{Q}$.
\end{lemma} 
 Our main result is the next theorem, which gives a more explicit formula to $\proba (\Tau^Y_b \leq T | W_T,v)$.
\begin{theorem}
\label{th_approximation_random_oneboundary}
Under \textbf{Assumption B},
we have
\begin{eqnarray}
\nonumber
\proba (\Tau^Y_b \leq T | W_T,v) & = & \exp{\big(-\alpha W_T + \frac{1}{2} \int_0^T \theta_s^2 ds\big)} \\& &\times \esp_{\mathbb{Q}} \big[\mathbf{1}_{\{\Tau^Y_b \leq T \}}  \exp{\big(-\widetilde{\alpha} \widetilde{W}\big)} | W_T,v \big] \label{eq_nonapproximation_random_oneboundary}.
\end{eqnarray}
\end{theorem}
 We first calculate $\mathbb{Q} (\Tau^Y_b \leq T | W_T,v)$.
\begin{lemma}
\label{prop_condproba_random_oneboundary1}
Under \textbf{Assumption B}, we have
\begin{eqnarray}
\label{propeq_condproba_random_oneboundary1}
\mathbb{Q} (\Tau^Y_b \leq T | W_T,v) & = &  \exp{\Big(-\frac{2b(b-Y_T)}{T}\Big)}\mathbf{1}_{\{Y_T \leq b\}} + \mathbf{1}_{\{Y_T > b\}}.
\end{eqnarray}
\end{lemma}
 The next theorem gives a formula based on a theoretical assumption (\ref{approximation_assumption_random_oneboundary}).
\begin{theorem}
\label{th_approximation2_random_oneboundary}
We assume that \textbf{Assumption B} and the following assumption 
\begin{eqnarray}
\label{approximation_assumption_random_oneboundary}
\esp_{\mathbb{Q}} \big[\mathbf{1}_{\{\Tau^Y_b \leq T \}}  \exp{\big(-\widetilde{\alpha} \widetilde{W}\big)} | W_T ,v\big] \\ \nonumber = \esp_{\mathbb{Q}} \big[\mathbf{1}_{\{\Tau^Y_b \leq T \}}  | W_T,v \big] \esp_{\mathbb{Q}} \big[  \exp{\big(-\widetilde{\alpha} \widetilde{W}\big)} | W_T,v \big]
\end{eqnarray}
holds. Then, we have
\begin{eqnarray}
\nonumber \proba (\Tau^Y_b \leq T | W_T,v) & = & \exp{\big(-\alpha W_T + \frac{1}{2} \int_0^T \theta_s^2 ds\big)} \\ \nonumber && \times \Big(\exp{\Big(-\frac{2b(b-Y_T)}{T}\Big)}\mathbf{1}_{\{Y_T \leq b\}} + \mathbf{1}_{\{Y_T > b\}}\Big) \\ & & \label{eq_approximation_random_oneboundary} \times \exp{\big(\widetilde{\alpha} \int_0^T \widetilde{\theta}_s \theta_s ds\big)} \mathcal{L}_N(\widetilde{\alpha}).
\end{eqnarray}
\end{theorem}
 Finally, we get $P_b^Y(T)$ in the next corollary by integrating $\proba (\Tau^Y_b \leq T | W_T,v)$ with respect to the value of $(W_T,v)$. We define the arrival space and cdf of $v$ as respectively $\Pi_v$ and $P_v$. Moreover,  we define $y_u$, $y_b$, $y_\theta$, etc. following the above definitions when integrating with respect to $y \in \Pi_v$.
\begin{corollary}
\label{th_proba_random_oneboundary1}
Under \textbf{Assumption B}, we have
\begin{eqnarray}
\nonumber 
P_b^Y(T) & = & 1 - \phi \Big(\frac{b-u_T}{\sqrt{T}}\Big) \\ \nonumber & &+ \int_{-\infty}^{b-u_T} \int_{\Pi_v} \frac{1}{\sqrt{2 \pi T}} \exp{\Big(-\frac{x^2}{2T}\Big)} \exp{\big(-y_\alpha x + \frac{1}{2} \int_0^T y_{\theta,s}^2 ds\big)} \\ \label{eq_proba_random_oneboundary} &&\times \esp_{\mathbb{Q}} \big[\mathbf{1}_{\{\Tau^Y_b \leq T \}}  \exp{\big(-\widetilde{\alpha} \widetilde{W}\big)} | W_T =x,v=y \big] dx dP_v(y).
\end{eqnarray}
If we further assume (\ref{approximation_assumption_random_oneboundary}), we have
\begin{eqnarray}
\nonumber 
P_b^Y(T) & = & 1 - \phi \Big(\frac{b-u_T}{\sqrt{T}}\Big) \\ \nonumber && + \int_{-\infty}^{b-u_T} \int_{\Pi_v} \frac{1}{\sqrt{2 \pi T}} \exp{\Big(-\frac{x^2}{2T}\Big)} \exp{\big(-y_\alpha x + \frac{1}{2} \int_0^T y_{\theta,s}^2 ds\big)}  \\ & & \label{eq_proba_approximation_random_oneboundary} \times p_v(y)\exp{\Big(-\frac{2y_b(y_b-y_{u,T}-x)}{T}\Big)} \\ & & \nonumber \times \exp{\big(y_{\widetilde{\alpha}} \int_0^T y_{\widetilde{\theta},s} y_{\theta,s} ds\big)} \mathcal{L}_N(y_{\widetilde{\alpha}}) dx d P_v(y).
\end{eqnarray}
\end{corollary}
\subsection{Two-sided time-varying boundary case}
In this section, we consider the case when the two-sided boundary and the drift are nonrandom time-varying and the variance is nonrandom.

We first give the definition of the set of two-sided boundary functions.
\begin{definition}\label{defboundaryset_nonrandom_twoboundary}
We define the set of two-sided boundary functions as $\mathcal{I} = \mathcal{G} \times \mathcal{G}.$
\end{definition}
 We now give the definition of the FPT to a two-sided boundary.
\begin{definition}\label{defFPT_nonrandom_twoboundary}
We define the FPT of an $\mathbf{F}$-adapted continuous process $Z$ to a two-sided boundary
$(g,h) \in \mathcal{I}$ satisfying
$g_0 \leq Z_0 \leq h_0$ as
\begin{eqnarray}
\label{TgZdef_nonrandom_twoboundary}
\Tau_{g,h}^Z = \inf \{t \in \reels^+ \text{ s.t. } Z_t \geq g_t \text{ or } Z_t \leq h_t \}.
\end{eqnarray}
\end{definition}
 We have that $Z$ is a continuous and 
$\mathbf{F}$-adapted stochastic process and $\inf \{t \in \reels^+ \text{ s.t. } Z_t \geq g_t \text{ or } Z_t \leq h_t \} = \inf \{t \in \reels^+ \text{ s.t. } Z_t \in G\}$ where $G= \{(t,u) \in \reels^+ \times \reels \text{ s.t. } u \geq g_t \text{ or } u \leq h_t\}$ is an open subset of $\reels^2$. Thus, the FPT $\Tau_{g,h}^{Z}$ is an $\mathbf{F}$-stopping time by Theorem I.1.27 (p. 7) in Jacod et al. (\citeyear{JacodLimit2003}). We can rewrite the boundary crossing probability $P_{g,h}^Z$ as the cdf of $\Tau_{g,h}^Z$, i.e., 
\begin{eqnarray}
\label{PgZdef_nonrandom_twoboundary}
P_{g,h}^Z(t)= \proba (\Tau^Z_{g,h} \leq t) \text{ for any } t \geq 0.
\end{eqnarray}
We assume that $g_0 < \mu_0 < h_0$. To apply the Girsanov theorem to the two-sided boundary case, we cannot use two different drifts since the process $Z_t$ is unique in Definition \ref{defFPT_nonrandom_twoboundary}. Thus, we have to restrict the class of boundary functions as we will assume that the deviation of $g$ from its starting value is equal to the sum of the deviation of $h$ from its starting value and a linear term, i.e., there exists $\beta \in \reels$ such that $h_t - h_0 = g_t - g_0 + \beta t$. Thus, we can rewrite the FPT to a two-sided time-varying boundary as an equivalent FPT to a two-sided boundary with one constant boundary and one linear boundary. More specifically, if we define the new drift as $u_t=\frac{\mu_t-\mu_0-g_t+g_0}{\sigma}$, the new process as $Y_t = u_t + W_t$, the new constant boundary as $b=\frac{\mu_0-g_0}{\sigma}$ and the new linear boundary as $c_t=\frac{h_0-\mu_0+\beta t}{\sigma}$, we observe that the FPT (\ref{TgZdef_nonrandom_twoboundary}) may be rewritten as $\Tau_{g,h}^Z = \Tau_{b,c}^Y$.

\begin{proof}[\textbf{Assumption C}]
We assume that $u \neq 0$, i.e., that there exists a $t \in [0,T]$ such that $u_t \neq 0$. We also assume that $u$ is absolutely continuous on $[0,T]$, i.e., there exists a nonrandom function $\theta : [0,T] \rightarrow \reels$ with $u_t=\int_0^t \theta_s ds$. Finally, we assume that $\int_0^T \theta_s^2 ds < \infty$.
\phantom\qedhere
\end{proof}
 By \textbf{Assumption C}, $M$ satisfies Novikov's condition and thus is a positive martingale. We embed this result and its implications on an equivalent probability measure $\mathbb{Q}$ by the Girsanov theorem in the following lemma.

\begin{lemma} \label{lemma_girsanov_nonrandom_twoboundary} Under \textbf{Assumption C}, we have that $M$ is a positive martingale. Thus, we can consider an equivalent probability measure $\mathbb{Q}$ such that the Radon-Nikodym derivative is defined as $\frac{d \mathbb{Q}}{d \proba}=M_T$. Finally, $Y$ is a standard Wiener process under $\mathbb{Q}$.
\end{lemma}
 The next proposition reexpresses $\proba (\Tau^Y_{b,c} \leq T | W_T)$ under $\mathbb{Q}$. We define $\overline{W}_t$ as
\begin{eqnarray}
\label{overlineWdef_nonrandom_twoboundary}
\overline{W}_t = \int_0^t \theta_s dW_s.
\end{eqnarray}
\begin{proposition}
\label{th_condproba_nonrandom_twoboundary}
Under \textbf{Assumption C}, we have
\begin{eqnarray}
\label{eq_condproba_nonrandom_twoboundary0}
\proba (\Tau^Y_{b,c} \leq T | W_T) & = & \esp_{\mathbb{Q}} \big[\mathbf{1}_{\{\Tau^Y_{b,c} \leq T \}} M_T^{-1} | W_T \big].
\end{eqnarray}
This can be reexpressed as
\begin{eqnarray}
\label{eq_condproba_nonrandom_twoboundary}
\proba (\Tau^Y_{b,c} \leq T | W_T) & = & \esp_{\mathbb{Q}} \big[ M_T^{-1} \esp_\mathbb{Q} \big[\mathbf{1}_{\{\Tau^Y_{b,c} \leq T \}} | W_T,\overline{W}_T \big] | W_T \big].
\end{eqnarray}
\end{proposition}
 To obtain a more explicit formula, we elementary decompose the Radon-Nikodym derivative inverse $M_T^{-1}$ as the product of a $\sigma(W_T)$-measurable random variable and a random variable independent from $W_T$. We define the correlation under $\proba$ between $W_T$ and $\overline{W}_T$ as $\rho$, i.e., $\rho = \operatorname{Cor}_\proba(W_T, \overline{W}_T)$. 
\begin{lemma} \label{lemma_representation_nonrandom_twoboundary} Under \textbf{Assumption C}, we have that $\frac{\overline{W}_T}{\sqrt{\int_0^T \theta_s^2 ds}}$ is a standard normal random variable under $\proba$. We can also show that $\rho  = \frac{\int_0^T \theta_s ds}{T\int_0^T \theta_s^2 ds}$. Moreover, there exists a standard normal random variable $\widetilde{W}$ under $\proba$ which is independent of $W_T$ and such that $\overline{W}_T$ when normalized can be reexpressed as
\begin{eqnarray}
\label{def_Wtilde_nonrandom_twoboundary}
\frac{\overline{W}_T}{\sqrt{\int_0^T \theta_s^2 ds}} = \rho \frac{W_T}{\sqrt{T}} + \sqrt{1-\rho^2} \widetilde{W}.
\end{eqnarray}
This can be reexpressed as
\begin{eqnarray}
\label{def_Wtilde0_nonrandom_twoboundary}
\overline{W}_T = \alpha W_T + \widetilde{\alpha} \widetilde{W},
\end{eqnarray}
where $\alpha = \rho\sqrt{T^{-1}\int_0^T \theta_s^2 ds}$ and $\widetilde{\alpha} = \sqrt{(1- \rho^2) \int_0^T \theta_s^2 ds}$. If we define $\widetilde{\theta}_t = \frac{\theta_s - \alpha}{\widetilde{\alpha}}$, we can reexpress $\widetilde{W}$ as
\begin{eqnarray}
\label{def_Wtilde1_nonrandom_twoboundary}
\widetilde{W} & = & \int_0^T \widetilde{\theta}_s dW_s.
\end{eqnarray}
Finally, $\widetilde{W}+\int_0^T \widetilde{\theta}_s \theta_s ds$ is a standard normal variable under $\mathbb{Q}$.
\end{lemma} 
 Our main result is the next theorem. 
\begin{theorem}
\label{th_approximation_nonrandom_twoboundary}
Under \textbf{Assumption C}, we have
\begin{eqnarray}
\label{eq_nonapproximation_nonrandom_twoboundary}
\proba (\Tau^Y_{b,c} \leq T | W_T) & = & \exp{\big(-\alpha W_T + \frac{1}{2} \int_0^T \theta_s^2 ds\big)} \\ \nonumber && \times \esp_{\mathbb{Q}} \big[\mathbf{1}_{\{\Tau^Y_{b,c} \leq T \}}  \exp{\big(-\widetilde{\alpha} \widetilde{W}\big)} | W_T \big].
\end{eqnarray}
\end{theorem}
 We first calculate $\mathbb{Q} (\Tau^Y_{b,c} \leq T | W_T)$ whose explicit formula is given in the following lemma. This reexpresses Anderson (\citeyear{anderson1960modification}) (Theorem 4.2, pp. 178-179) under $\mathbb{Q}$, which considers the linear case $\mu_t=\gamma t$, $g_t= a t + b$, $h_t= c t + d$ and $\sigma = 1$ under $\proba$. He obtains that
\begin{eqnarray*}
\proba (\Tau^Z_{g,h} \leq T | W_T = x) = \sum_{j=1}^{\infty} p^Z_{g,h}(j | x) \mathbf{1}_{\{x \in [h_T-\mu_T,g_T-\mu_T]\}} + \mathbf{1}_{\{x \notin [h_T-\mu_T,g_T-\mu_T]\}}.
\end{eqnarray*}
Here, $p^Z_{g,h}(j | x)$ is defined as
\begin{eqnarray*}
p^Z_{g,h}(j | x) &=& \exp{\Big(-\frac{2}{T}(j \delta_0 + h_0)(j \delta_T + (h_T - \mu_T) -x) \Big)} \\&&+ \exp{\Big(-\frac{2j}{T}(j \delta_0 \delta_T+\delta_0((h_T-\mu_T) -x) -\delta_T h_0) \Big)}\\ && + \exp{\Big(-\frac{2}{T}(j \delta_0 - g_0)(j \delta_T - ((g_T - \mu_T) -x)) \Big)} \\&&+ \exp{\Big(-\frac{2j}{T}(j \delta_0 \delta_T-\delta_0((g_T-\mu_T) -x) +\delta_T g_0) \Big)}
\end{eqnarray*}
for any $j\in \naturels_*$, $x \in [h_T-\mu_T,g_T-\mu_T]$ and the difference between $g$ and $h$ is defined as $\delta_t = \delta_t (g,h)=g_t - h_t$ for any $t \in [0,T]$.

\begin{lemma}
\label{prop_condproba_nonrandom_twoboundary1}
Under \textbf{Assumption C}, we have
\begin{eqnarray}
\label{propeq_condproba_nonrandom_twoboundary1}
\mathbb{Q} (\Tau^Y_{b,c} \leq T | W_T) & = &  \sum_{j=1}^{\infty} q^Y_{b,c}(j | Y_T) \mathbf{1}_{\{Y_T \in [c_T,b_T]\}} + \mathbf{1}_{\{Y_T \notin [c_T,b_T]\}}.
\end{eqnarray}
Here, $q^Y_{b,c}(j | x)$ is defined as
\begin{eqnarray*}
q^Y_{b,c}(j | x) &=& \exp{\Big(-\frac{2}{T}(j \delta_0 + c_0)(j \delta_T + (c_T -x) )\Big)} \\&&+ \exp{\Big(-\frac{2j}{T}(j \delta_0 \delta_T+\delta_0(c_T -x) -\delta_T c_0) \Big)}\\ && + \exp{\Big(-\frac{2}{T}(j \delta_0 - b_0)(j \delta_T - (b_T -x)) \Big)} \\&&+ \exp{\Big(-\frac{2j}{T}(j \delta_0 \delta_T-\delta_0((b_T -x) +\delta_T b_0) \Big)},
\end{eqnarray*}
for any $j\in \naturels_*$, $x\in[c_T,b_T]$ and $\delta_t =\delta_t (b,c)$.
\end{lemma}
 The next theorem gives a formula based on a strong theoretical assumption (\ref{approximation_assumption_nonrandom_twoboundary}).
\begin{theorem}
\label{th_approximation2_nonrandom_twoboundary}
We assume that \textbf{Assumption C} and the following assumption 
\begin{eqnarray}
\label{approximation_assumption_nonrandom_twoboundary}
\esp_{\mathbb{Q}} \big[\mathbf{1}_{\{\Tau^Y_{b,c} \leq T \}}  \exp{\big(-\widetilde{\alpha} \widetilde{W}\big)} | W_T \big] &=& \esp_{\mathbb{Q}} \big[\mathbf{1}_{\{\Tau^Y_{b,c} \leq T \}}  | W_T \big] \\ \nonumber && \times \esp_{\mathbb{Q}} \big[  \exp{\big(-\widetilde{\alpha} \widetilde{W}\big)} | W_T \big] 
\end{eqnarray}
holds. Then, we have
\begin{eqnarray}
\label{eq_approximation_nonrandom_twoboundary}
\proba (\Tau^Y_{b,c} \leq T | W_T) & = & \exp{\big(-\alpha W_T + \frac{1}{2} \int_0^T \theta_s^2 ds\big)} \\ \nonumber && \Big(\sum_{j=1}^{\infty} q^Y_{b,c}(j | Y_T) \mathbf{1}_{\{Y_T \in [c_T,b_T]\}} \\ &&+ \mathbf{1}_{\{Y_T \notin [c_T,b_T]\}} \Big) \nonumber  \exp{\big(\widetilde{\alpha} \int_0^T \widetilde{\theta}_s \theta_s ds\big)} \mathcal{L}_N(\widetilde{\alpha}).
\end{eqnarray}
\end{theorem}
 Finally, we get $P_{b,c}^Y(T)$ in the next corollary by integrating $\proba (\Tau^Y_{b,c} \leq T | W_T)$ with respect to the value of $W_T$.
\begin{corollary}
\label{th_proba_nonrandom_twoboundary1}
Under \textbf{Assumption C}, we have
\begin{eqnarray}
\nonumber 
P_{b,c}^Y(T) & = & 1 - \phi (\frac{b_T-u_T}{\sqrt{T}}) + \phi (\frac{c_T-u_T}{\sqrt{T}}) \\&& \nonumber + \int_{c_T-u_T}^{b_T-u_T} \frac{1}{\sqrt{2 \pi T}} \exp{\Big(-\frac{x^2}{2T}\Big)} \exp{\big(-\alpha x + \frac{1}{2} \int_0^T \theta_s^2 ds\big)} \\ \label{eq_proba_nonrandom_twoboundary} &&\times\esp_{\mathbb{Q}} \big[\mathbf{1}_{\{\Tau^Y_{b,c} \leq T \}}  \exp{\big(-\widetilde{\alpha} \widetilde{W}\big)} | W_T =x \big] dx.
\end{eqnarray}
If we further assume (\ref{approximation_assumption_nonrandom_twoboundary}), we have
\begin{eqnarray}
\nonumber 
P_{b,c}^Y(T) & = & 1 - \phi (\frac{b_T-u_T}{\sqrt{T}}) + \phi (\frac{c_T-u_T}{\sqrt{T}}) \\&& \nonumber + \int_{c_T-u_T}^{b_T-u_T} \frac{1}{\sqrt{2 \pi T}} \exp{\Big(-\frac{x^2}{2T}\Big)} \exp{\big(-\alpha x + \frac{1}{2} \int_0^T \theta_s^2 ds\big)}  \\ & & \nonumber \times \Big(\sum_{j=1}^{\infty} q^{x+u_T}_{b,c}(j | x+u_T) \mathbf{1}_{\{x \in [c_T-u_T,b_T-u_T]\}} \\ && \nonumber + \mathbf{1}_{\{x \notin [c_T-u_T,b_T-u_T]\}}\Big) \\&& \label{eq_proba_approximation_nonrandom_twoboundary}\times \exp{\big(\widetilde{\alpha} \int_0^T \widetilde{\theta}_s \theta_s ds\big)} \mathcal{L}_N(\widetilde{\alpha}) dx.
\end{eqnarray}
\end{corollary}
 In the particular case when the boundaries and the drift are linear, and the standard deviation is equal to unity, Corollary \ref{th_proba_nonrandom_twoboundary1} reduces to Anderson (\citeyear{anderson1960modification}) (Theorem 4.3, p. 180).

\section{Proofs}
Our proofs rely on an elementary application of the Girsanov theorem. 
\subsection{One-sided time-varying boundary case}
We start with the proofs in the case when the one-sided boundary and the drift are nonrandom time-varying and the variance is nonrandom. 

The proof of Proposition \ref{th_condproba_nonrandom_oneboundary} is based on Lemma \ref{lemma_girsanov} and its consequence in the particular case $X= \mathbf{1}_{\{\Tau^Y_b \leq T \}} \mathbf{1}_{E_T}$ in which $E_T$ is a $\sigma(W_T)$-measurable event. 
\begin{proof}[Proof of Proposition \ref{th_condproba_nonrandom_oneboundary}] 
By definition of the conditional probability, Equation (\ref{eq_condproba_nonrandom_oneboundary0}) can be rewritten formally as 
\begin{eqnarray}
\label{eq_condexp_nonrandom_oneboundary0}
\esp_\proba \big[\mathbf{1}_{\{\Tau^Y_b \leq T \}} | W_T \big]& =& \esp_{\mathbb{Q}} \big[\mathbf{1}_{\{\Tau^Y_b \leq T \}} M_T^{-1} | W_T \big].
\end{eqnarray}
For any $\sigma(W_T)$-measurable event $E_T$, we can use a change of probability in the expectation by Lemma \ref{lemma_girsanov} along with \textbf{Assumption A} and we obtain that
\begin{eqnarray}
\label{eq_condexp_nonrandom_oneboundary00}
\esp_\proba \big[\mathbf{1}_{\{\Tau^Y_b \leq T \}} \mathbf{1}_{E_T}\big]& =& \esp_{\mathbb{Q}} \big[\mathbf{1}_{\{\Tau^Y_b \leq T \}} M_T^{-1} \mathbf{1}_{E_T} \big].
\end{eqnarray}
We can deduce Equation (\ref{eq_condexp_nonrandom_oneboundary0}) from Equation (\ref{eq_condexp_nonrandom_oneboundary00}) by definition of the conditional expectation. By definition of the conditional probability, Equation (\ref{eq_condproba_nonrandom_oneboundary}) can be rewritten formally as 
\begin{eqnarray}
\label{eq_condexp_nonrandom_oneboundary}
\esp_\proba \big[\mathbf{1}_{\{\Tau^Y_b \leq T \}} | W_T \big]& =& \esp_{\mathbb{Q}} \big[ M_T^{-1} \esp_\mathbb{Q} \big[\mathbf{1}_{\{\Tau^Y_b \leq T \}} | W_T,\overline{W}_T \big] | W_T \big].
\end{eqnarray}
By definition of the conditional expectation, we can deduce what follows. If we can show that for any $E_T$ which is $\sigma(W_T)$-measurable that 
\begin{eqnarray}
\label{eq_condexpevent_nonrandom_oneboundary}
\esp_\proba \big[\mathbf{1}_{\{\Tau^Y_b \leq T \}} \mathbf{1}_{E_T} \big] \\ \nonumber = \esp_\proba \Big[ \esp_{\mathbb{Q}} \big[ M_T^{-1} \esp_\mathbb{Q} \big[\mathbf{1}_{\{\Tau^Y_b \leq T \}} | W_T,\overline{W}_T \big] | W_T \big] \mathbf{1}_{E_T} \Big],
\end{eqnarray}
then Equation (\ref{eq_condexp_nonrandom_oneboundary}) holds. Let $E_T$ be a $\sigma(W_T)$-measurable event. By Lemma \ref{lemma_girsanov} along with \textbf{Assumption A}, we can use a change of probability in the expectation and we obtain that
\begin{eqnarray}
\label{proof0}
\esp_\proba \big[\mathbf{1}_{\{\Tau^Y_b \leq T \}} \mathbf{1}_{E_T} \big] & = & \esp_{\mathbb{Q}} \big[\mathbf{1}_{\{\Tau^Y_b \leq T \}} \mathbf{1}_{E_T} M_T^{-1}\big].
\end{eqnarray}
Then, we have by the law of total expectation that
\begin{eqnarray}
\label{proof1}
\esp_{\mathbb{Q}} \big[\mathbf{1}_{\{\Tau^Y_b \leq T \}} \mathbf{1}_{E_T} M_T^{-1}\big]  & = & \esp_{\mathbb{Q}} \big[ \esp_{\mathbb{Q}} \big[\mathbf{1}_{\{\Tau^Y_b \leq T \}} \mathbf{1}_{E_T} M_T^{-1} | W_T, \overline{W}_T \big] \big].
\end{eqnarray}
Since $\mathbf{1}_{E_T}$ and $M_T^{-1}$ are $\sigma(W_T,\overline{W}_T)$-measurable random variables, we can pull them out of the conditional expectation and deduce that
\begin{eqnarray}
\label{proof2}
& & \esp_{\mathbb{Q}} \big[ \esp_{\mathbb{Q}} \big[\mathbf{1}_{\{\Tau^Y_b \leq T \}} \mathbf{1}_{E_T} M_T^{-1} | W_T, \overline{W}_T\big] \big] \\ \nonumber & = & \esp_{\mathbb{Q}} \big[ \mathbf{1}_{E_T} M_T^{-1}\esp_{\mathbb{Q}} \big[\mathbf{1}_{\{\Tau^Y_b \leq T \}} | W_T, \overline{W}_T \big] \big].
\end{eqnarray}
If we use Equations (\ref{proof0})-(\ref{proof1})-(\ref{proof2}), we can deduce that Equation (\ref{eq_condexpevent_nonrandom_oneboundary}) holds.
\end{proof}
 In what follows, we give the proof of Lemma \ref{lemma_representation}. It is based on the fact that $(W_T,\overline{W}_T)$ 
is a centered normal random vector under $\proba$ so that we can rewrite $\overline{W}_T$ as $\overline{W}_T = \alpha W_T + \widetilde{\alpha} \widetilde{W}$.  
\begin{proof}[Proof of Lemma \ref{lemma_representation}] 
By \textbf{Assumption A}, we have that $\overline{W}_T$ is well-defined and $\int_0^T \theta_s^2 ds < \infty$ thus we can deduce that $\frac{\overline{W}_T}{\sqrt{\int_0^T \theta_s^2 ds}}$ is a standard normal random variable under $\proba$. Since $(W_T,\overline{W}_T)$ is a centered normal random vector under $\proba$, there exists a standard normal random variable $\widetilde{W}$ under $\proba$ which is independent of $W_T$ and such that Equation (\ref{def_Wtilde}) holds. Then, we can calculate that
\begin{eqnarray*}
\rho & = & \operatorname{Cor}_\proba(W_T,\int_0^T \theta_s dW_s)\\
& = & \frac{\operatorname{Cov}_\proba(W_T,\int_0^T \theta_s dW_s)}{\operatorname{Var}_\proba (W_T) \operatorname{Var}_\proba (\int_0^T \theta_s dW_s)}\\
& = & \frac{\int_0^T \theta_s ds}{T\int_0^T \theta_s^2 ds}.
\end{eqnarray*} 
Here, we use the definition of $\rho$ and Equation (\ref{overlineWdef}) in the first equality, the fact that $\theta \neq 0$ by \textbf{Assumption A} in the second equality, and the It\^{o} isometry in the last equality. Equation (\ref{def_Wtilde0}) can be deduced directly from Equation (\ref{def_Wtilde}). Moreover, we can reexpress $\widetilde{W}$ as
\begin{eqnarray*}
\widetilde{W} & = &\frac{1}{\widetilde{\alpha}} (\overline{W}_T - \alpha W_T)\\
& = & \int_0^T \frac{\theta_s - \alpha}{\widetilde{\alpha}} dW_s\\
&=& \int_0^T \widetilde{\theta}_s dW_s.
\end{eqnarray*}
Here, we use Equation (\ref{def_Wtilde0}) in the first equality, Equation (\ref{overlineWdef}) in the second equality and the definition of $\widetilde{\theta}_t$ in the last equality. Finally, we can deduce that $\widetilde{W}+\int_0^T \widetilde{\theta}_s \theta_s ds$ is a standard normal variable under $\mathbb{Q}$ by its expression (\ref{def_Wtilde1}) and since by Lemma \ref{lemma_girsanov} along with \textbf{Assumption A}, $Y$ is a Wiener process under $\mathbb{Q}$.
\end{proof}
 We provide now the proof of Theorem \ref{th_approximation_nonrandom_oneboundary} which is based on Lemma \ref{lemma_representation}. 
\begin{proof}[Proof of Theorem \ref{th_approximation_nonrandom_oneboundary}] We can reexpress $M_T$ as
\begin{eqnarray}
\nonumber M_T & = & \exp{\big(\int_0^T \theta_s dW_s - \frac{1}{2} \int_0^T \theta_s^2 ds\big)}\\ \nonumber
& = & \exp{\big(\overline{W}_T - \frac{1}{2} \int_0^T \theta_s^2 ds\big)}\\ \nonumber
& = & \exp{\big(\alpha W_T + \widetilde{\alpha} \widetilde{W} - \frac{1}{2} \int_0^T \theta_s^2 ds\big)}\\
\label{M_representation}
& = & \exp{\big(\alpha W_T - \frac{1}{2} \int_0^T \theta_s^2 ds\big)} \exp{\big(\widetilde{\alpha} \widetilde{W}\big)}.
\end{eqnarray}
Here, we use Equation (\ref{defM}) in the first equality, Equation (\ref{overlineWdef}) in the second equality, Equation (\ref{def_Wtilde0}) from Lemma \ref{lemma_representation} in the third equality and algebraic manipulation in the last equality. Then, we have
\begin{eqnarray}
\nonumber\proba (\Tau^Y_b \leq T | W_T) & = & \esp_{\mathbb{Q}} \big[\mathbf{1}_{\{\Tau^Y_b \leq T \}} M_T^{-1} | W_T \big]\\\nonumber & = & \esp_{\mathbb{Q}} \big[\mathbf{1}_{\{\Tau^Y_b \leq T \}} \exp{\big(-\alpha W_T + \frac{1}{2} \int_0^T \theta_s^2 ds\big)} \\ \nonumber & & \times \exp{\big(-\widetilde{\alpha} \widetilde{W}\big)} | W_T \big]\\\nonumber & = & \exp{\big(-\alpha W_T + \frac{1}{2} \int_0^T \theta_s^2 ds\big)} \\ \nonumber & & \times \esp_{\mathbb{Q}} \big[\mathbf{1}_{\{\Tau^Y_b \leq T \}}  \exp{\big(-\widetilde{\alpha} \widetilde{W}\big)} | W_T \big].
\end{eqnarray}
Here, we use Equation (\ref{eq_condproba_nonrandom_oneboundary0}) from Proposition \ref{th_condproba_nonrandom_oneboundary}
along with \textbf{Assumption A} in the first equality, Equation (\ref{M_representation}) in the second equality, and the fact that $W_T$ is a $\sigma(W_T)$-measurable random variable in the third equality. Thus, we have shown Equation (\ref{eq_nonapproximation_nonrandom_oneboundary}).
\end{proof}
 We now give the proof which reexpresses Malmquist (\citeyear{malmquist1954certain}) (Theorem 1, p. 526) under $\mathbb{Q}$.
\begin{proof}[Proof of Lemma \ref{prop_condproba_nonrandom_oneboundary1}] 
By definition of the conditional probability, Equation (\ref{propeq_condproba_nonrandom_oneboundary1}) can be rewritten formally as 
\begin{eqnarray}
\label{eq_condexp_nonrandom_oneboundary1}
\esp_\mathbb{Q} \big[\mathbf{1}_{\{\Tau^Y_b \leq T \}} | W_T \big]& =& \exp{\Big(-\frac{2b(b-Y_T)}{T}\Big)}\mathbf{1}_{\{Y_T \leq b\}} + \mathbf{1}_{\{Y_T > b\}}.
\end{eqnarray}
By Lemma \ref{lemma_girsanov} along with \textbf{Assumption A}, $Y$ is a Wiener process under $\mathbb{Q}$. Then, we have by Malmquist (\citeyear{malmquist1954certain}) (Theorem 1, p. 526) that Equation (\ref{eq_condexp_nonrandom_oneboundary1}) holds. 
\end{proof}
 We provide now the proof of Theorem \ref{th_approximation2_nonrandom_oneboundary}, which is based on Lemma \ref{prop_condproba_nonrandom_oneboundary1}. 
\begin{proof}[Proof of Theorem \ref{th_approximation2_nonrandom_oneboundary}] We have
\begin{eqnarray}
\nonumber \esp_{\mathbb{Q}} \big[\mathbf{1}_{\{\Tau^Y_b \leq T \}}  \exp{\big(-\widetilde{\alpha} \widetilde{W}\big)} | W_T \big]  & = & \esp_{\mathbb{Q}} \big[\mathbf{1}_{\{\Tau^Y_b \leq T \}}  | W_T \big] \esp_{\mathbb{Q}} \big[  \exp{\big(-\widetilde{\alpha} \widetilde{W}\big)} | W_T \big]\\ \nonumber& = & \Big(\exp{\Big(-\frac{2b(b-Y_T)}{T}\Big)}\mathbf{1}_{\{Y_T \leq b\}} + \mathbf{1}_{\{Y_T > b\}}\Big)\\ \label{proof_approximation} && \times \esp_{\mathbb{Q}} \big[  \exp{\big(-\widetilde{\alpha} \widetilde{W}\big)} | W_T \big],
\end{eqnarray}
where we use Assumption (\ref{approximation_assumption}) in the first equality, Equation (\ref{propeq_condproba_nonrandom_oneboundary1}) from Lemma \ref{prop_condproba_nonrandom_oneboundary1} along with \textbf{Assumption A} in the second equality. Finally, we have
\begin{eqnarray}
\nonumber \esp_{\mathbb{Q}} \big[  \exp{\big(-\widetilde{\alpha} \widetilde{W}\big)} | W_T \big] & = & \esp_{\mathbb{Q}} \big[  \exp{\big(-\widetilde{\alpha} \widetilde{W}\big)} \big]\\ \nonumber & = & \exp{\big(\widetilde{\alpha} \int_0^T \widetilde{\theta}_s \theta_s ds\big)} \\ \nonumber & & \times \esp_{\mathbb{Q}} \big[  \exp{\big(-\widetilde{\alpha} \big(\widetilde{W}+\int_0^T \widetilde{\theta}_s \theta_s ds\big)\big)} \big]\\ \nonumber & = & \exp{\big(\widetilde{\alpha} \int_0^T \widetilde{\theta}_s \theta_s ds\big)} \esp_{\proba} \big[  \exp{\big(-\widetilde{\alpha} N \big)} \big]\\ \label{proof_approximation0} & = & \exp{\big(\widetilde{\alpha} \int_0^T \widetilde{\theta}_s \theta_s ds\big)} \mathcal{L}_N(\widetilde{\alpha}).
\end{eqnarray}
Here, we use the fact that $\widetilde{W}$ is independent from $W_T$ in the first equality, algebraic manipulation in the second equality, the fact that $\widetilde{W}+\int_0^T \widetilde{\theta}_s \theta_s ds$ is a standard normal variable under $\mathbb{Q}$ by Lemma \ref{lemma_representation} along with \textbf{Assumption A} in the third equality, and Equation (\ref{def_laplace_transform}) in the last equality. We can deduce Equation (\ref{eq_approximation_nonrandom_oneboundary}) from Equations (\ref{eq_nonapproximation_nonrandom_oneboundary}), (\ref{proof_approximation}) and (\ref{proof_approximation0}).
\end{proof}
 Finally, the proof of Corollary \ref{th_proba_nonrandom_oneboundary1} follows the steps of Equations (3) in Wang et al. (\citeyear{wang1997boundary}) (p. 55).  
\begin{proof}[Proof of Corollary \ref{th_proba_nonrandom_oneboundary1}] 
We can calculate that
\begin{eqnarray*}
P_b^Y(T) & = & \int_{-\infty}^{\infty} \proba (\Tau^Y_b \leq T | W_T = x) \frac{1}{\sqrt{2 \pi T}} \exp{\Big(-\frac{x^2}{2T}\Big)}  dx\\ & = & 1 - \phi (\frac{b-u_T}{\sqrt{T}}) \\ & & \nonumber+ \int_{-\infty}^{b-u_T} \proba (\Tau^Y_b \leq T | W_T = x) \frac{1}{\sqrt{2 \pi T}} \exp{\Big(-\frac{x^2}{2T}\Big)}  dx\\ & = & 1 - \phi (\frac{b-u_T}{\sqrt{T}}) \\ & & \nonumber + \int_{-\infty}^{b-u_T} \frac{1}{\sqrt{2 \pi T}} \exp{\Big(-\frac{x^2}{2T}\Big)} \exp{\big(-\alpha x + \frac{1}{2} \int_0^T \theta_s^2 ds\big)} \\ &&\times\esp_{\mathbb{Q}} \big[\mathbf{1}_{\{\Tau^Y_b \leq T \}}  \exp{\big(-\widetilde{\alpha} \widetilde{W}\big)} | W_T =x \big] dx.
\end{eqnarray*}
Here, we use Equation (\ref{PgZdef}) and regular conditional probability in the first equality, the fact that $\proba (\Tau^Y_b \leq T | W_T = x)=1$ for any $x \geq b-u_T$ in the second equality, and Equation (\ref{eq_nonapproximation_nonrandom_oneboundary}) in the third equality. We have thus shown Equation (\ref{eq_proba_nonrandom_oneboundary}). Equation (\ref{eq_proba_approximation_nonrandom_oneboundary}) can be shown following the same first two equalities and using Assumption (\ref{eq_approximation_nonrandom_oneboundary}) in the third equality.
\end{proof}

\subsection{One-sided stochastic boundary process case}
We continue with the proofs in the case when the one-sided boundary and the drift are stochastic processes and the variance is random. 

The elementary idea in the proofs of this section is to condition by both $W_T$ and $v$, i.e., to derive results of the form $\proba (\Tau^Y_b \leq T | W_T,v)$. The proof of Proposition \ref{th_condproba_random_oneboundary} is based on Lemma \ref{lemma_girsanov_random_oneboundary} and its consequence in the particular case $X= \mathbf{1}_{\{\Tau^Y_b \leq T \}} \mathbf{1}_{E_T}$, in which $E_T$ is a $\sigma(W_T,v)$-measurable event.
\begin{proof}[Proof of Proposition \ref{th_condproba_random_oneboundary}] 
By definition of the conditional probability, Equation (\ref{eq_condproba_random_oneboundary0}) can be rewritten formally as 
\begin{eqnarray}
\label{eq_condexp_random_oneboundary0}
\esp_\proba \big[\mathbf{1}_{\{\Tau^Y_b \leq T \}} | W_T,v \big]& =& \esp_{\mathbb{Q}} \big[\mathbf{1}_{\{\Tau^Y_b \leq T \}} M_T^{-1} | W_T ,v\big].
\end{eqnarray}
For any $\sigma(W_T,v)$-measurable event $E_T$, we can use a change of probability in the expectation by Lemma \ref{lemma_girsanov_random_oneboundary} along with \textbf{Assumption B} and we obtain that
\begin{eqnarray}
\label{eq_condexp_random_oneboundary00}
\esp_\proba \big[\mathbf{1}_{\{\Tau^Y_b \leq T \}} \mathbf{1}_{E_T}\big]& =& \esp_{\mathbb{Q}} \big[\mathbf{1}_{\{\Tau^Y_b \leq T \}} M_T^{-1} \mathbf{1}_{E_T} \big].
\end{eqnarray}
We can deduce Equation (\ref{eq_condexp_random_oneboundary0}) from Equation (\ref{eq_condexp_random_oneboundary00}) by definition of the conditional expectation. By definition of the conditional probability, Equation (\ref{eq_condproba_random_oneboundary}) can be rewritten formally as 
\begin{eqnarray}
\label{eq_condexp_random_oneboundary}
\esp_\proba \big[\mathbf{1}_{\{\Tau^Y_b \leq T \}} | W_T ,v\big]\\ \nonumber = \esp_{\mathbb{Q}} \big[ M_T^{-1} \esp_\mathbb{Q} \big[\mathbf{1}_{\{\Tau^Y_b \leq T \}} | W_T,\overline{W}_T,v \big] | W_T,v \big].
\end{eqnarray}
By definition of the conditional expectation, we can deduce what follows. If we can show that for any $E_T$ which is $\sigma(W_T,v)$-measurable that 
\begin{eqnarray}
\label{eq_condexpevent_random_oneboundary}
\esp_\proba \big[\mathbf{1}_{\{\Tau^Y_b \leq T \}} \mathbf{1}_{E_T} \big] \\ \nonumber = \esp_\proba \Big[ \esp_{\mathbb{Q}} \big[ M_T^{-1} \esp_\mathbb{Q} \big[\mathbf{1}_{\{\Tau^Y_b \leq T \}} | W_T,\overline{W}_T,v \big] | W_T,v \big] \mathbf{1}_{E_T} \Big],
\end{eqnarray}
then Equation (\ref{eq_condexp_random_oneboundary}) holds. Let $E_T$ be a $\sigma(W_T,v)$-measurable event. By Lemma \ref{lemma_girsanov_random_oneboundary} along with \textbf{Assumption B}, we can use a change of probability in the expectation and we obtain that
\begin{eqnarray}
\label{proof0_random_oneboundary}
\esp_\proba \big[\mathbf{1}_{\{\Tau^Y_b \leq T \}} \mathbf{1}_{E_T} \big] & = & \esp_{\mathbb{Q}} \big[\mathbf{1}_{\{\Tau^Y_b \leq T \}} \mathbf{1}_{E_T} M_T^{-1}\big].
\end{eqnarray}
Then, we have by the law of total expectation that
\begin{eqnarray}
\label{proof1_random_oneboundary}
\esp_{\mathbb{Q}} \big[\mathbf{1}_{\{\Tau^Y_b \leq T \}} \mathbf{1}_{E_T} M_T^{-1}\big]   = \\ \nonumber \esp_{\mathbb{Q}} \big[ \esp_{\mathbb{Q}} \big[\mathbf{1}_{\{\Tau^Y_b \leq T \}} \mathbf{1}_{E_T} M_T^{-1} | W_T, \overline{W}_T ,v\big] \big].
\end{eqnarray}
Since $\mathbf{1}_{E_T}$ and $M_T^{-1}$ are $\sigma(W_T, \overline{W}_T ,v)$-measurable random variables, we can pull them out of the conditional expectation and deduce that
\begin{eqnarray}
\label{proof2_random_oneboundary}
\esp_{\mathbb{Q}} \big[ \esp_{\mathbb{Q}} \big[\mathbf{1}_{\{\Tau^Y_b \leq T \}} \mathbf{1}_{E_T} M_T^{-1} | W_T, \overline{W}_T,v\big] \big] \\ \nonumber =  \esp_{\mathbb{Q}} \big[ \mathbf{1}_{E_T} M_T^{-1}\esp_{\mathbb{Q}} \big[\mathbf{1}_{\{\Tau^Y_b \leq T \}} | W_T, \overline{W}_T,v \big] \big].
\end{eqnarray}
If we use Equations (\ref{proof0_random_oneboundary})-(\ref{proof1_random_oneboundary})-(\ref{proof2_random_oneboundary}), we can deduce that Equation (\ref{eq_condexpevent_random_oneboundary}) holds.
\end{proof}
 In what follows, we give the proof of Lemma \ref{lemma_representation_random_oneboundary}. It is based on the fact that $(W_T,\frac{\overline{W}_T}{\sqrt{\int_0^T \theta_s^2 ds}})$ 
is a centered normal random vector under $\proba$. 
\begin{proof}[Proof of Lemma \ref{lemma_representation_random_oneboundary}] 
By \textbf{Assumption B}, we can deduce that $ 0 < \int_0^T \theta_s^2 ds < \infty$ a.s.. Thus, we can normalize $\overline{W}_T$ by $\sqrt{\int_0^T \theta_s^2 ds}$ a.s. and we have that $\frac{\overline{W}_T}{\sqrt{\int_0^T \theta_s^2 ds}}$ is a mixed normal random variable a.s. by definition. We have that its conditional mean under $\proba$ is a.s. equal to
\begin{eqnarray}
\nonumber
\esp_\proba \Big[\frac{\int_0^T \theta_s dW_s}{\sqrt{\int_0^T \theta_s^2 ds}}\Big| v \Big] \nonumber
&=& \frac{1}{\sqrt{\int_0^T \theta_s^2 ds}}\esp_\proba \Big[\int_0^T \theta_s dW_s \Big| v \Big]\\ \label{proof_mean_random_oneboundary} &=& 0.
\end{eqnarray}
Here, we use the fact that $\frac{1}{\sqrt{\int_0^T \theta_s^2 ds}}$ is $\sigma(v)$-measurable in the first equality, and the fact that $\int_0^T \theta_s dW_s$ is a.s. a martingale since $ \int_0^T \theta_s^2 ds < \infty$ a.s. in the second equality. We also have that its conditional variance under $\proba$ is a.s. equal to
\begin{eqnarray}
\nonumber
\operatorname{Var}_\proba \Big(\frac{\int_0^T \theta_s dW_s}{\sqrt{\int_0^T \theta_s^2 ds}}\Big| v\Big) &=& \esp_\proba \Big[\Big(\frac{\int_0^T \theta_s dW_s}{\sqrt{\int_0^T \theta_s^2 ds}}\Big)^2\Big| v \Big]\\\nonumber
&=& \frac{1}{\int_0^T \theta_s^2 ds}\esp_\proba \Big[\Big(\int_0^T \theta_s dW_s\Big)^2\Big| v \Big]\\ \label{proof_variance_random_oneboundary} &=& 1.
\end{eqnarray}
Here, we use Equation (\ref{proof_mean_random_oneboundary}) in the first equality, the fact that $\frac{1}{\sqrt{\int_0^T \theta_s^2 ds}}$ is $\sigma(v)$-measurable in the second equality, and the It\^{o} isometry in the third equality. Since its conditional mean and conditional variance are nonrandom, we obtain that its mean under $\proba$ is equal to $\esp_\proba \Big[\frac{\int_0^T \theta_s dW_s}{\sqrt{\int_0^T \theta_s^2 ds}}\Big]
= \esp_\proba \Big[\esp_\proba \Big[\frac{\int_0^T \theta_s dW_s}{\sqrt{\int_0^T \theta_s^2 ds}}\Big| v \Big]\Big]= 0$ by the law of total expectation and Equation (\ref{proof_mean_random_oneboundary}), and similarly that its variance is equal to 1 by the law of total expectation and Equation (\ref{proof_variance_random_oneboundary}).
Thus, we have that $\frac{\overline{W}_T}{\sqrt{\int_0^T \theta_s^2 ds}}$ is a standard normal random variable under $\proba$. Since $(W_T,\frac{\overline{W}_T}{\sqrt{\int_0^T \theta_s^2 ds}})$ is a centered normal random vector under $\proba$, there exists a standard normal random variable $\widetilde{W}$ under $\proba$ which is independent of $W_T$ and such that Equation (\ref{def_Wtilde_random_oneboundary}) holds. Then, we can calculate that the covariance between $W_T$ and $\frac{\overline{W}_T}{\sqrt{\int_0^T \theta_s^2 ds}}$ under $\proba$ is equal to
\begin{eqnarray}
\nonumber
\operatorname{Cov}_\proba \Big(W_T,\frac{\overline{W}_T}{\sqrt{\int_0^T \theta_s^2 ds}} \Big) &=& \operatorname{Cov}_\proba \Big(W_T,\frac{\int_0^T \theta_s dW_s}{\sqrt{\int_0^T \theta_s^2 ds}} \Big)\\ \nonumber&=& \esp_\proba \Big[W_T\frac{\int_0^T \theta_s dW_s}{\sqrt{\int_0^T \theta_s^2 ds}} \Big]\\ \nonumber&=& \esp_\proba \Big[\esp \Big[W_T\frac{\int_0^T \theta_s dW_s}{\sqrt{\int_0^T \theta_s^2 ds}} \Big| v \Big]\Big]\\\nonumber &=& \esp_\proba \Big[\frac{1}{\sqrt{\int_0^T \theta_s^2 ds}}\esp \Big[W_T\int_0^T \theta_s dW_s \Big| v \Big]\Big]\\ \label{proof_cov_random_oneboundary} &=& \esp_\proba \Big[\frac{\int_0^T \theta_s ds }{\sqrt{\int_0^T \theta_s^2 ds}}\Big].
\end{eqnarray}
Here, we use Equation (\ref{overlineWdef}) in the first equality, Equation (\ref{proof_mean_random_oneboundary}) in the second equality, the law of total expectation in the third equality, the fact that $\frac{1}{\sqrt{\int_0^T \theta_s^2 ds}}$ is $\sigma(v)$-measurable in the fourth equality, and the It\^{o} isometry in the last equality. Now, we can calculate that the correlation between $W_T$ and $\frac{\overline{W}_T}{\sqrt{\int_0^T \theta_s^2 ds}}$ under $\proba$ is equal to
\begin{eqnarray*}
\rho & = & \operatorname{Cor}_\proba\Big(W_T,\frac{\int_0^T \theta_s dW_s}{\sqrt{\int_0^T \theta_s^2 ds}}\Big)\\
& = & \frac{\operatorname{Cov}_\proba\Big(W_T,\frac{\int_0^T \theta_s dW_s}{\sqrt{\int_0^T \theta_s^2 ds}}\Big)}{\operatorname{Var}_\proba (W_T) \operatorname{Var}_\proba \Big(\frac{\int_0^T \theta_s dW_s}{\sqrt{\int_0^T \theta_s^2 ds}}\Big)}\\
& = & \frac{1}{T}\esp_\proba \Big[\frac{\int_0^T \theta_s ds }{\sqrt{\int_0^T \theta_s^2 ds}}\Big].
\end{eqnarray*} 
Here, we use the definition of $\rho$ and Equation (\ref{overlineWdef}) in the first equality, and Equations (\ref{proof_variance_random_oneboundary}) and (\ref{proof_cov_random_oneboundary}) in the last equality. Equation (\ref{def_Wtilde0_random_oneboundary}) can be deduced directly from Equation (\ref{def_Wtilde_random_oneboundary}). Moreover, we can reexpress $\widetilde{W}$ as
\begin{eqnarray*}
\widetilde{W} & = &\frac{1}{\widetilde{\alpha}} (\overline{W}_T - \alpha W_T)\\
& = & \int_0^T \frac{\theta_s - \alpha}{\widetilde{\alpha}} dW_s\\
&=& \int_0^T \widetilde{\theta}_s dW_s,
\end{eqnarray*}
where we use Equation (\ref{def_Wtilde0_random_oneboundary}) in the first equality, Equation (\ref{overlineWdef}) in the second equality and the definition of $\widetilde{\theta}_t$ in the last equality. Moreover, we can deduce that $\widetilde{W}+\int_0^T \widetilde{\theta}_s \theta_s ds$ is a standard normal variable under $\mathbb{Q}$. This is due to its expression (\ref{def_Wtilde1_random_oneboundary}) and since by Lemma \ref{lemma_girsanov_random_oneboundary} along with \textbf{Assumption B}, $Y$ is a Wiener process under $\mathbb{Q}$. Finally, $\mathcal{D}(\widetilde{W}+\int_0^T \widetilde{\theta}_s \theta_s ds | v)$ is standard normal under $\mathbb{Q}$ by Equation (\ref{def_Wtilde1_random_oneboundary}).
\end{proof}
 We provide now the proof of Theorem \ref{th_approximation_random_oneboundary}, which is based on Lemma \ref{lemma_representation_random_oneboundary}.
\begin{proof}[Proof of Theorem \ref{th_approximation_random_oneboundary}] We can reexpress $M_T$ as
\begin{eqnarray}
\nonumber M_T & = & \exp{\big(\int_0^T \theta_s dW_s - \frac{1}{2} \int_0^T \theta_s^2 ds\big)}\\ \nonumber
& = & \exp{\big(\overline{W}_T - \frac{1}{2} \int_0^T \theta_s^2 ds\big)}\\ \nonumber
& = & \exp{\big(\alpha W_T + \widetilde{\alpha} \widetilde{W} - \frac{1}{2} \int_0^T \theta_s^2 ds\big)}\\
\label{M_representation}
& = & \exp{\big(\alpha W_T - \frac{1}{2} \int_0^T \theta_s^2 ds\big)} \exp{\big(\widetilde{\alpha} \widetilde{W}\big)}.
\end{eqnarray}
Here, we use Equation (\ref{defM}) in the first equality, Equation (\ref{overlineWdef}) in the second equality, Equation (\ref{def_Wtilde0_random_oneboundary}) from Lemma \ref{lemma_representation_random_oneboundary} in the third equality, and algebraic manipulation in the last equality. Then, we have
\begin{eqnarray}
\nonumber\proba (\Tau^Y_b \leq T | W_T,v) & = & \esp_{\mathbb{Q}} \big[\mathbf{1}_{\{\Tau^Y_b \leq T \}} M_T^{-1} | W_T,v \big]\\\nonumber & = & \esp_{\mathbb{Q}} \big[\mathbf{1}_{\{\Tau^Y_b \leq T \}} \exp{\big(-\alpha W_T + \frac{1}{2} \int_0^T \theta_s^2 ds\big)} \\ && \nonumber \times \exp{\big(-\widetilde{\alpha} \widetilde{W}\big)} | W_T,v \big]\\\nonumber & = & \exp{\big(-\alpha W_T + \frac{1}{2} \int_0^T \theta_s^2 ds\big)} \\ && \nonumber \times\esp_{\mathbb{Q}} \big[\mathbf{1}_{\{\Tau^Y_b \leq T \}}  \exp{\big(-\widetilde{\alpha} \widetilde{W}\big)} | W_T,v \big].
\end{eqnarray}
Here, we use Equation (\ref{eq_condproba_random_oneboundary0}) from Proposition \ref{th_condproba_random_oneboundary}
along with \textbf{Assumption B} in the first equality, Equation (\ref{M_representation}) in the second equality, the fact that $W_T$ and $\theta_t$ for any $t \in [0,T]$ are $\sigma(W_T,v)$-measurable random variables in the third equality. Thus, we have shown Equation (\ref{eq_nonapproximation_random_oneboundary}).
\end{proof}
 We now give the proof of Lemma \ref{prop_condproba_random_oneboundary1}.
\begin{proof}[Proof of Lemma \ref{prop_condproba_random_oneboundary1}] 
By definition of the conditional probability, Equation (\ref{propeq_condproba_random_oneboundary1}) can be rewritten formally as 
\begin{eqnarray}
\label{eq_condexp_random_oneboundary1}
\esp_\mathbb{Q} \big[\mathbf{1}_{\{\Tau^Y_b \leq T \}} | W_T,v \big]& =& \exp{\Big(-\frac{2b(b-Y_T)}{T}\Big)}\mathbf{1}_{\{Y_T \leq b\}} + \mathbf{1}_{\{Y_T > b\}}.
\end{eqnarray}
By Lemma \ref{lemma_girsanov_random_oneboundary} along with \textbf{Assumption B}, $Y$ is a Wiener process under $\mathbb{Q}$. Then, we have by \cite{malmquist1954certain} (Theorem 1, p. 526) that Equation (\ref{eq_condexp_random_oneboundary1}) holds. 
\end{proof}
 We provide now the proof of Theorem \ref{th_approximation2_random_oneboundary}, which is based on Lemma \ref{prop_condproba_random_oneboundary1}.
\begin{proof}[Proof of Theorem \ref{th_approximation2_random_oneboundary}] We have
\begin{eqnarray}
\nonumber \esp_{\mathbb{Q}} \big[\mathbf{1}_{\{\Tau^Y_b \leq T \}}  \exp{\big(-\widetilde{\alpha} \widetilde{W}\big)} | W_T,v \big]  & = & \esp_{\mathbb{Q}} \big[\mathbf{1}_{\{\Tau^Y_b \leq T \}}  | W_T,v \big] \\ && \nonumber \times \esp_{\mathbb{Q}} \big[  \exp{\big(-\widetilde{\alpha} \widetilde{W}\big)} | W_T,v \big]\\ \nonumber& = & \Big(\exp{\Big(-\frac{2b(b-Y_T)}{T}\Big)}\mathbf{1}_{\{Y_T \leq b\}} \\ && \nonumber + \mathbf{1}_{\{Y_T > b\}}\Big)\\ && \label{proof_approximation_random_oneboundary} \times \esp_{\mathbb{Q}} \big[  \exp{\big(-\widetilde{\alpha} \widetilde{W}\big)} | W_T,v \big],
\end{eqnarray}
where we use Assumption (\ref{approximation_assumption_random_oneboundary}) in the first equality, Equation (\ref{propeq_condproba_random_oneboundary1}) from Lemma \ref{prop_condproba_random_oneboundary1} along with \textbf{Assumption B} in the second equality. Finally, we have
\begin{eqnarray}
\nonumber \esp_{\mathbb{Q}} \big[  \exp{\big(-\widetilde{\alpha} \widetilde{W}\big)} | W_T,v \big] & = & \esp_{\mathbb{Q}} \big[  \exp{\big(-\widetilde{\alpha} \widetilde{W}\big)} | v\big]\\ \nonumber & = & \exp{\big(\widetilde{\alpha} \int_0^T \widetilde{\theta}_s \theta_s ds\big)} \\ && \nonumber \times\esp_{\mathbb{Q}} \big[  \exp{\big(-\widetilde{\alpha} \big(\widetilde{W}+\int_0^T \widetilde{\theta}_s \theta_s ds\big)\big)} | v \big]\\ \nonumber & = & \exp{\big(\widetilde{\alpha} \int_0^T \widetilde{\theta}_s \theta_s ds\big)} \\ && \nonumber \times\esp_{\mathbb{Q}} \big[  \exp{\big(-\widetilde{\alpha} \big(\widetilde{W}+\int_0^T \widetilde{\theta}_s \theta_s ds\big)\big)} \big]\\ \nonumber & = & \exp{\big(\widetilde{\alpha} \int_0^T \widetilde{\theta}_s \theta_s ds\big)} \esp_{\proba} \big[  \exp{\big(-\widetilde{\alpha} N \big)} \big]\\ \label{proof_approximation0_random_oneboundary} & = & \exp{\big(\widetilde{\alpha} \int_0^T \widetilde{\theta}_s \theta_s ds\big)} \mathcal{L}_N(\widetilde{\alpha}).
\end{eqnarray}
Here, we use the fact that $\widetilde{W}$ is independent from $W_T$ in the first equality, the fact that $\theta_t$ and $\widetilde{\theta}_t$ for any $t \in [0,T]$ are $\sigma(v)$-measurable random variables in the second equality, the fact that $\mathcal{D}(\widetilde{W}+\int_0^T \widetilde{\theta}_s \theta_s ds | v)$ is standard normal under $\mathbb{Q}$ by Lemma \ref{lemma_representation_random_oneboundary} along with \textbf{Assumption B} in the third equality, the fact that $\widetilde{W}+\int_0^T \widetilde{\theta}_s \theta_s ds$ is a standard normal variable under $\mathbb{Q}$ by Lemma \ref{lemma_representation_random_oneboundary} along with \textbf{Assumption B} in the fourth equality, and Equation (\ref{def_laplace_transform}) in the last equality. We can deduce Equation (\ref{eq_approximation_random_oneboundary}) from Equations (\ref{eq_nonapproximation_random_oneboundary}), (\ref{proof_approximation_random_oneboundary}) and (\ref{proof_approximation0_random_oneboundary}).
\end{proof}
 Finally, we get $P_b^Y(T)$ in the next corollary, by integrating $\proba (\Tau^Y_b \leq T | W_T,v)$ with respect to the value of $(W_T,v)$.
\begin{proof}[Proof of Corollary \ref{th_proba_random_oneboundary1}] 
We can calculate that
\begin{eqnarray*}
P_b^Y(T) & = & \int_{-\infty}^{\infty} \int_{\Pi_v} \proba (\Tau^Y_b \leq T | W_T = x,v=y) \frac{1}{\sqrt{2 \pi T}} \exp{\Big(-\frac{x^2}{2T}\Big)}  dx dP_v(y)\\ & = & 1 - \phi (\frac{b-u_T}{\sqrt{T}}) \\ & & + \int_{-\infty}^{b-u_T} \int_{\Pi_v} \proba (\Tau^Y_b \leq T | W_T = x,v=y) \\ & & \nonumber \times \frac{1}{\sqrt{2 \pi T}} \exp{\Big(-\frac{x^2}{2T}\Big)} dx dP_v(y)\\ & = & 1 - \phi (\frac{b-u_T}{\sqrt{T}}) + \\ & & \nonumber \int_{-\infty}^{b-u_T} \int_{\Pi_v} \frac{1}{\sqrt{2 \pi T}} \exp{\Big(-\frac{x^2}{2T}\Big)} \exp{\big(-y_\alpha x + \frac{1}{2} \int_0^T y_{\theta,s}^2 ds\big)} \\  &&\times \esp_{\mathbb{Q}} \big[\mathbf{1}_{\{\Tau^Y_b \leq T \}}  \exp{\big(-\widetilde{\alpha} \widetilde{W}\big)} | W_T =x,v=y \big] dx dP_v(y).
\end{eqnarray*}
Here, we use Equation (\ref{PgZdef_random_oneboundary}), regular conditional probability and the fact that $W_T$ and $v$ are independent in the first equality, the fact that $\proba (\Tau^Y_b \leq T | W_T = x)=1$ for any $x \geq b-u_T$ in the second equality, and Equation (\ref{eq_nonapproximation_random_oneboundary}) in the third equality. We have thus shown Equation (\ref{eq_proba_random_oneboundary}). Equation (\ref{eq_proba_approximation_random_oneboundary}) can be shown following the same first two equalities and using Equation (\ref{eq_approximation_random_oneboundary}) in the third equality.
\end{proof}

\subsection{Two-sided time-varying boundary case}
In this section, we consider the proofs in the case when the two-sided boundary and the drift are nonrandom time-varying and the variance is nonrandom.

The proof of Proposition \ref{th_condproba_nonrandom_twoboundary} is based on Lemma \ref{lemma_girsanov_nonrandom_twoboundary}. 
\begin{proof}[Proof of Proposition \ref{th_condproba_nonrandom_twoboundary}] 
By definition of the conditional probability, Equation (\ref{eq_condproba_nonrandom_twoboundary0}) can be rewritten formally as 
\begin{eqnarray}
\label{eq_condexp_nonrandom_twoboundary0}
\esp_\proba \big[\mathbf{1}_{\{\Tau^Y_{b,c} \leq T \}} | W_T \big]& =& \esp_{\mathbb{Q}} \big[\mathbf{1}_{\{\Tau^Y_{b,c} \leq T \}} M_T^{-1} | W_T \big].
\end{eqnarray}
For any $\sigma(W_T)$-measurable event $E_T$, we can use a change of probability in the expectation by Lemma \ref{lemma_girsanov_nonrandom_twoboundary} along with \textbf{Assumption C} and we obtain that
\begin{eqnarray}
\label{eq_condexp_nonrandom_twoboundary00}
\esp_\proba \big[\mathbf{1}_{\{\Tau^Y_{b,c} \leq T \}} \mathbf{1}_{E_T}\big]& =& \esp_{\mathbb{Q}} \big[\mathbf{1}_{\{\Tau^Y_{b,c} \leq T \}} M_T^{-1} \mathbf{1}_{E_T} \big].
\end{eqnarray}
We can deduce Equation (\ref{eq_condexp_nonrandom_twoboundary0}) from Equation (\ref{eq_condexp_nonrandom_twoboundary00}) by definition of the conditional expectation. By definition of the conditional probability, Equation (\ref{eq_condproba_nonrandom_twoboundary}) can be rewritten formally as 
\begin{eqnarray}
\label{eq_condexp_nonrandom_twoboundary}
\esp_\proba \big[\mathbf{1}_{\{\Tau^Y_{b,c} \leq T \}} | W_T \big]& =& \esp_{\mathbb{Q}} \big[ M_T^{-1} \esp_\mathbb{Q} \big[\mathbf{1}_{\{\Tau^Y_{b,c} \leq T \}} | W_T,\overline{W}_T \big] | W_T \big].
\end{eqnarray}
By definition of the conditional expectation, we can deduce what follows. If we can show that for any $E_T$ which is $\sigma(W_T)$-measurable that 
\begin{eqnarray}
\label{eq_condexpevent_nonrandom_twoboundary}
\esp_\proba \big[\mathbf{1}_{\{\Tau^Y_{b,c} \leq T \}} \mathbf{1}_{E_T} \big] \\ \nonumber = \esp_\proba \Big[ \esp_{\mathbb{Q}} \big[ M_T^{-1} \esp_\mathbb{Q} \big[\mathbf{1}_{\{\Tau^Y_{b,c} \leq T \}} | W_T,\overline{W}_T \big] | W_T \big] \mathbf{1}_{E_T} \Big],
\end{eqnarray}
then Equation (\ref{eq_condexp_nonrandom_twoboundary}) holds. Let $E_T$ be a $\sigma(W_T)$-measurable event. By Lemma \ref{lemma_girsanov_nonrandom_twoboundary} along with \textbf{Assumption C}, we obtain that
\begin{eqnarray}
\label{proof0_nonrandom_twoboundary}
\esp_\proba \big[\mathbf{1}_{\{\Tau^Y_{b,c} \leq T \}} \mathbf{1}_{E_T} \big] & = & \esp_{\mathbb{Q}} \big[\mathbf{1}_{\{\Tau^Y_{b,c} \leq T \}} \mathbf{1}_{E_T} M_T^{-1}\big].
\end{eqnarray}
Then we have by the law of total expectation that
\begin{eqnarray}
\label{proof1_nonrandom_twoboundary}
\esp_{\mathbb{Q}} \big[\mathbf{1}_{\{\Tau^Y_{b,c} \leq T \}} \mathbf{1}_{E_T} M_T^{-1}\big]   \\ \nonumber =  \esp_{\mathbb{Q}} \big[ \esp_{\mathbb{Q}} \big[\mathbf{1}_{\{\Tau^Y_{b,c} \leq T \}} \mathbf{1}_{E_T} M_T^{-1} | W_T, \overline{W}_T \big] \big].
\end{eqnarray}
Since $\mathbf{1}_{E_T}$ and $M_T^{-1}$ are $\sigma(W_T,\overline{W}_T)$-measurable random variables, we can pull them out of the conditional expectation and deduce that
\begin{eqnarray}
\label{proof2_nonrandom_twoboundary}
& & \esp_{\mathbb{Q}} \big[ \esp_{\mathbb{Q}} \big[\mathbf{1}_{\{\Tau^Y_{b,c} \leq T \}} \mathbf{1}_{E_T} M_T^{-1} | W_T, \overline{W}_T\big] \big] \\ \nonumber & = & \esp_{\mathbb{Q}} \big[ \mathbf{1}_{E_T} M_T^{-1}\esp_{\mathbb{Q}} \big[\mathbf{1}_{\{\Tau^Y_{b,c} \leq T \}} | W_T, \overline{W}_T \big] \big].
\end{eqnarray}
If we use Equations (\ref{proof0_nonrandom_twoboundary})-(\ref{proof1_nonrandom_twoboundary})-(\ref{proof2_nonrandom_twoboundary}), we can deduce that Equation (\ref{eq_condexpevent_nonrandom_twoboundary}) holds.
\end{proof}
 In what follows, we give the proof of Lemma \ref{lemma_representation_nonrandom_twoboundary}.  
\begin{proof}[Proof of Lemma \ref{lemma_representation_nonrandom_twoboundary}] 
By \textbf{Assumption C}, we can deduce that $\frac{\overline{W}_T}{\sqrt{\int_0^T \theta_s^2 ds}}$ is a standard normal random variable under $\proba$. Since $(W_T,\overline{W}_T)$ is a centered normal random vector under $\proba$, there exists a standard normal random variable $\widetilde{W}$ under $\proba$ which is independent of $W_T$ and such that Equation (\ref{def_Wtilde_nonrandom_twoboundary}) holds. Using the same arguments from the proof of Lemma \ref{lemma_representation}, we can calculate that 
\begin{eqnarray*}
\rho & = & \frac{\int_0^T \theta_s ds}{T\int_0^T \theta_s^2 ds}.
\end{eqnarray*} 
Moreover, we can reexpress $\widetilde{W}$ as
\begin{eqnarray*}
\widetilde{W} & = & \int_0^T \widetilde{\theta}_s dW_s.
\end{eqnarray*}
Finally, we can deduce that $\widetilde{W}+\int_0^T \widetilde{\theta}_s \theta_s ds$ is a standard normal variable under $\mathbb{Q}$. This is due to its expression (\ref{def_Wtilde1_nonrandom_twoboundary}) and since by Lemma \ref{lemma_girsanov_nonrandom_twoboundary} along with \textbf{Assumption C}, $Y$ is a Wiener process under $\mathbb{Q}$.
\end{proof}
 We provide now the proof of Theorem \ref{th_approximation_nonrandom_twoboundary}, which is based on Lemma \ref{lemma_representation_nonrandom_twoboundary}. 
\begin{proof}[Proof of Theorem \ref{th_approximation_nonrandom_twoboundary}] By the same arguments from the proof of Theorem \ref{th_approximation_nonrandom_oneboundary}, we can reexpress $M_T$ as
\begin{eqnarray*}
 M_T 
& = & \exp{\big(\alpha W_T - \frac{1}{2} \int_0^T \theta_s^2 ds\big)} \exp{\big(\widetilde{\alpha} \widetilde{W}\big)}.
\end{eqnarray*}
Then, we have
\begin{eqnarray}
\nonumber\proba (\Tau^Y_{b,c} \leq T | W_T) & = & \exp{\big(-\alpha W_T + \frac{1}{2} \int_0^T \theta_s^2 ds\big)} \\ \nonumber & & \times \esp_{\mathbb{Q}} \big[\mathbf{1}_{\{\Tau^Y_{b,c} \leq T \}}  \exp{\big(-\widetilde{\alpha} \widetilde{W}\big)} | W_T \big].
\end{eqnarray}
Thus, we have shown Equation (\ref{eq_nonapproximation_nonrandom_twoboundary}).
\end{proof}
 We now give the proof which reexpresses Anderson (\citeyear{anderson1960modification}) (Theorem 4.2, pp. 178-179) under $\mathbb{Q}$. 
\begin{proof}[Proof of Lemma \ref{prop_condproba_nonrandom_twoboundary1}] 
By definition of the conditional probability, Equation (\ref{propeq_condproba_nonrandom_twoboundary1}) can be rewritten formally as 
\begin{eqnarray}
\label{eq_condexp_nonrandom_twoboundary1}
\esp_\mathbb{Q} \big[\mathbf{1}_{\{\Tau^Y_b \leq T \}} | W_T \big]& =& \sum_{j=1}^{\infty} q^Y_{b,c}(j | Y_T) \mathbf{1}_{\{Y_T \in [c_T,b_T]\}} + \mathbf{1}_{\{Y_T \notin [c_T,b_T]\}}.
\end{eqnarray}
By Lemma \ref{lemma_girsanov_nonrandom_twoboundary} along with \textbf{Assumption C}, $Y$ is a Wiener process under $\mathbb{Q}$. Then, we have by Anderson (\citeyear{anderson1960modification}) (Theorem 4.2, pp. 178-179) that Equation (\ref{eq_condexp_nonrandom_twoboundary1}) holds. 
\end{proof}
 We provide now the proof of Theorem \ref{th_approximation2_nonrandom_twoboundary}, which is based on Lemma \ref{prop_condproba_nonrandom_twoboundary1}. 
\begin{proof}[Proof of Theorem \ref{th_approximation2_nonrandom_twoboundary}] 
We have
\begin{eqnarray}
\nonumber \esp_{\mathbb{Q}} \big[\mathbf{1}_{\{\Tau^Y_{b,c} \leq T \}}  \exp{\big(-\widetilde{\alpha} \widetilde{W}\big)} | W_T \big]  & = & \esp_{\mathbb{Q}} \big[\mathbf{1}_{\{\Tau^Y_{b,c} \leq T \}}  | W_T \big] \esp_{\mathbb{Q}} \big[  \exp{\big(-\widetilde{\alpha} \widetilde{W}\big)} | W_T \big]\\ \label{proof_approximation_nonrandom_twoboundary} & = & \Big(\sum_{j=1}^{\infty} q^Y_{b,c}(j | Y_T) \mathbf{1}_{\{Y_T \in [c_T,b_T]\}} \\ \nonumber && + \mathbf{1}_{\{Y_T \notin [c_T,b_T]\}}\Big) \esp_{\mathbb{Q}} \big[  \exp{\big(-\widetilde{\alpha} \widetilde{W}\big)} | W_T \big],
\end{eqnarray}
where we use Assumption (\ref{approximation_assumption_nonrandom_twoboundary}) in the first equality, and Equation (\ref{propeq_condproba_nonrandom_twoboundary1}) from Lemma \ref{prop_condproba_nonrandom_twoboundary1} along with \textbf{Assumption C} in the second equality. Finally, we have
\begin{eqnarray}
\nonumber \esp_{\mathbb{Q}} \big[  \exp{\big(-\widetilde{\alpha} \widetilde{W}\big)} | W_T \big] & = & \esp_{\mathbb{Q}} \big[  \exp{\big(-\widetilde{\alpha} \widetilde{W}\big)} \big]\\ \nonumber & = & \exp{\big(\widetilde{\alpha} \int_0^T \widetilde{\theta}_s \theta_s ds\big)} \\ \nonumber & & \times \esp_{\mathbb{Q}} \big[  \exp{\big(-\widetilde{\alpha} \big(\widetilde{W}+\int_0^T \widetilde{\theta}_s \theta_s ds\big)\big)} \big]\\ \nonumber & = & \exp{\big(\widetilde{\alpha} \int_0^T \widetilde{\theta}_s \theta_s ds\big)} \esp_{\proba} \big[  \exp{\big(-\widetilde{\alpha} N \big)} \big]\\ \label{proof_approximation0_nonrandom_twoboundary} & = & \exp{\big(\widetilde{\alpha} \int_0^T \widetilde{\theta}_s \theta_s ds\big)} \mathcal{L}_N(\widetilde{\alpha}).
\end{eqnarray}
Here, we use the fact that $\widetilde{W}$ is independent from $W_T$ in the first equality, algebraic manipulation in the second equality, the fact that $\widetilde{W}+\int_0^T \widetilde{\theta}_s \theta_s ds$ is a standard normal variable under $\mathbb{Q}$ by Lemma \ref{lemma_representation_nonrandom_twoboundary} along with \textbf{Assumption C} in the third equality, and Equation (\ref{def_laplace_transform}) in the last equality. We can deduce Equation (\ref{eq_approximation_nonrandom_twoboundary}) from Equations (\ref{eq_nonapproximation_nonrandom_twoboundary}), (\ref{proof_approximation_nonrandom_twoboundary}) and (\ref{proof_approximation0_nonrandom_twoboundary}).
\end{proof}
 Finally, we get $P_{b,c}^Y(T)$ in the next proof, by integrating $\proba (\Tau^Y_{b,c} \leq T | W_T)$ with respect to the value of $W_T$.
\begin{proof}[Proof of Corollary \ref{th_proba_nonrandom_twoboundary1}] 
We can calculate that
\begin{eqnarray*}
P_{b,c}^Y(T) & = & \int_{-\infty}^{\infty} \proba (\Tau^Y_{b,c} \leq T | W_T = x) \frac{1}{\sqrt{2 \pi T}} \exp{\Big(-\frac{x^2}{2T}\Big)}  dx\\ & = & 1 - \phi (\frac{b_T-u_T}{\sqrt{T}}) + \phi (\frac{c_T-u_T}{\sqrt{T}}) \\&&+ \int_{c_T-u_T}^{b_T-u_T} \proba (\Tau^Y_{b,c} \leq T | W_T = x) \frac{1}{\sqrt{2 \pi T}} \exp{\Big(-\frac{x^2}{2T}\Big)}  dx\\ & = & 1 - \phi (\frac{b_T-u_T}{\sqrt{T}}) + \phi (\frac{c_T-u_T}{\sqrt{T}}) \\&&+ \int_{c_T-u_T}^{b_T-u_T} \frac{1}{\sqrt{2 \pi T}} \exp{\Big(-\frac{x^2}{2T}\Big)} \exp{\big(-\alpha x + \frac{1}{2} \int_0^T \theta_s^2 ds\big)} \\ &&\times\esp_{\mathbb{Q}} \big[\mathbf{1}_{\{\Tau^Y_{b,c} \leq T \}}  \exp{\big(-\widetilde{\alpha} \widetilde{W}\big)} | W_T =x \big] dx.
\end{eqnarray*}
Here, we use Equation (\ref{PgZdef_nonrandom_twoboundary}) and regular conditional probability in the first equality, the fact that $\proba (\Tau^Y_{b,c} \leq T | W_T = x)=1$ for any $x \geq b_T-u_T$ and any $x \leq c_T-u_T$ in the second equality, and Equation (\ref{eq_nonapproximation_nonrandom_twoboundary}) in the third equality. We have thus shown Equation (\ref{eq_proba_nonrandom_twoboundary}). Equation (\ref{eq_proba_approximation_nonrandom_twoboundary}) can be shown following the same first two equalities and using Equation (\ref{eq_approximation_nonrandom_twoboundary}) in the third equality.
\end{proof}



\begin{funding}
The author was supported in part by Japanese Society for the Promotion of Science Grants-in-Aid for Scientific Research (B) 23H00807 and Early-Career Scientists 20K13470. 
\end{funding}

\begin{supplement}
\stitle{Supplement A: Results in the two-sided stochastic boundary process case}
\sdescription{Supplement A gives the results in the two-sided stochastic boundary process case.}
\end{supplement}
 \begin{supplement}
\stitle{Supplement B: Proofs in the two-sided stochastic boundary process case}
\sdescription{Supplement B collects the proofs in the two-sided stochastic boundary process case.}
\end{supplement}

\newpage
\begin{appendix}
Supplement A gives the results in the two-sided stochastic boundary process case. Supplement B collects the proofs in the two-sided stochastic boundary process case.
\section{Results in the two-sided stochastic boundary process case}
In this appendix, we consider the case when the two-sided boundary and the drift are stochastic processes and the variance is random.

We first give the definition of the set of stochastic boundary processes.
\begin{definition}\label{defboundaryset_random_twoboundary}
We define the set of stochastic two-sided boundary processes as $\mathcal{J} = \reels^+ \times \Omega \rightarrow \reels^2$ such that for any $(g,h) \in \mathcal{J}$ and $\omega \in \Omega$ we have $(g,h)(\omega) \in \mathcal{I}$ as well as $g$ and $h$ are $\mathbf{F}$-adapted.
\end{definition}
 We now give the definition of the FPT.
\begin{definition}\label{defFPT_random_twoboundary}
We define the FPT of an $\mathbf{F}$-adapted continuous process $Z$ to the two-sided boundary 
$(g,h) \in \mathcal{J}$ satisfying $g_0 \leq Z_0 \leq h_0$ $\forall \omega \in \Omega$ as
\begin{eqnarray}
\label{TgZdef_random_twoboundary}
\Tau_{g,h}^Z = \inf \{t \in \reels^+ \text{ s.t. } Z_t \geq g_t \text{ or } Z_t \leq h_t \}.
\end{eqnarray}
\end{definition}
 We can rewrite $\Tau_{g,h}^Z$ as the infimum of two $\mathbf{F}$-stopping times, i.e., $\Tau_{g,h}^Z = \inf (\Tau_{h}^Z,\Tau_{-g}^{-Z})$. Thus, it is an $\mathbf{F}$-stopping time. We can rewrite the boundary crossing probability $P_{g,h}^Z$ as the cdf of $\Tau_{g,h}^Z$, i.e., 
\begin{eqnarray}
\label{PgZdef_random_twoboundary}
P_{g,h}^Z(t)= \proba (\Tau^Z_{g,h} \leq t) \text{ for any } t \geq 0.
\end{eqnarray}
We assume that $\mu$ is an $\mathbf{F}$-adapted stochastic process which satisfies $\proba(g_0 <\mu_0 < h_0) = 1$. We also assume that the variance $\sigma^2$ is time-invariant, random, and such that $\proba(\sigma^2 = 0)=0$. Finally, we assume that $v$ is independent of $W$ where $v$ is defined as $v=(g,h,\mu,\sigma)$. 
\begin{proof}[\textbf{Assumption D}]
We assume that $\proba( \exists t \in [0,T] \text{ s.t } u_t \neq 0) = 1$. We also assume that $u$ is absolutely continuous on $[0,T]$, i.e., there exists a stochastic process $\theta : [0,T] \times \Omega \rightarrow \reels$ with $u_t=\int_0^t \theta_s ds$, a.s.. Finally, we assume that $\esp [\exp{\big(\frac{1}{2} \int_0^T \theta_s^2 ds\big)}] < \infty$.
\phantom\qedhere
\end{proof}
 By \textbf{Assumption D}, $M$ satisfies Novikov's condition and thus is a positive martingale.
 \begin{lemma} \label{lemma_girsanov_random_twoboundary} Under \textbf{Assumption D}, we have that $M$ is a positive martingale. Thus, we can consider an equivalent probability measure $\mathbb{Q}$ such that the Radon-Nikodym derivative is defined as $\frac{d \mathbb{Q}}{d \proba}=M_T$. Finally, $Y$ is a standard Wiener process under $\mathbb{Q}$. \end{lemma}
 The elementary idea in this appendix is to condition by both $W_T$ and $v$, i.e., to derive results of the form $\proba (\Tau^Y_{b,c} \leq T | W_T,v)$. The next proposition reexpresses $\proba (\Tau^Y_{b,c} \leq T | W_T,v)$ under $\mathbb{Q}$. We define $\overline{W}_t$ as
\begin{eqnarray}
\label{overlineWdef_random_twoboundary}
\overline{W}_t = \int_0^t \theta_s dW_s.
\end{eqnarray}
\begin{proposition}
\label{th_condproba_random_twoboundary}
Under \textbf{Assumption D}, we have
\begin{eqnarray}
\label{eq_condproba_random_twoboundary0}
\proba (\Tau^Y_{b,c} \leq T | W_T,v) & = & \esp_{\mathbb{Q}} \big[\mathbf{1}_{\{\Tau^Y_{b,c} \leq T \}} M_T^{-1} | W_T ,v\big].
\end{eqnarray}
This can be reexpressed as
\begin{eqnarray}
\label{eq_condproba_random_twoboundary}
\proba (\Tau^Y_{b,c} \leq T | W_T,v)  =  \\ \nonumber \esp_{\mathbb{Q}} \big[ M_T^{-1} \esp_\mathbb{Q} \big[\mathbf{1}_{\{\Tau^Y_{b,c} \leq T \}} | W_T,\overline{W}_T,v \big] | W_T,v \big].
\end{eqnarray}
\end{proposition}
 We define the correlation under $\proba$ between $W_T$ and $\frac{\overline{W}_T}{\sqrt{\int_0^T \theta_s^2 ds}}$ as $\rho$, i.e., $\rho = \operatorname{Cor}_\proba(W_T, \frac{\overline{W}_T}{\sqrt{\int_0^T \theta_s^2 ds}})$.
\begin{lemma} \label{lemma_representation_random_twoboundary} Under \textbf{Assumption D}, we have that $\frac{\overline{W}_T}{\sqrt{\int_0^T \theta_s^2 ds}}$ is a standard normal random variable under $\proba$. We can also show that $\rho  = \frac{1}{T}\esp_\proba \Big[\frac{\int_0^T \theta_s ds }{\sqrt{\int_0^T \theta_s^2 ds}}\Big]$ a.s.. Moreover, there exists a standard normal random variable $\widetilde{W}$ under $\proba$, which is independent of $W_T$, and such that $\overline{W}_T$ when normalized can be reexpressed a.s. as
\begin{eqnarray}
\label{def_Wtilde_random_twoboundary}
\frac{\overline{W}_T}{\sqrt{\int_0^T \theta_s^2 ds}} = \rho \frac{W_T}{\sqrt{T}} + \sqrt{1-\rho^2} \widetilde{W}.
\end{eqnarray}
This can be reexpressed a.s. as
\begin{eqnarray}
\label{def_Wtilde0_random_twoboundary}
\overline{W}_T = \alpha W_T + \widetilde{\alpha} \widetilde{W},
\end{eqnarray}
where $\alpha = \rho\sqrt{T^{-1}\int_0^T \theta_s^2 ds}$ a.s. and $\widetilde{\alpha} = \sqrt{(1- \rho^2) \int_0^T \theta_s^2 ds}$ a.s.. If we define $\widetilde{\theta}_t = \frac{\theta_s - \alpha}{\widetilde{\alpha}}$, we can reexpress $\widetilde{W}$ a.s. as
\begin{eqnarray}
\label{def_Wtilde1_random_twoboundary}
\widetilde{W} & = & \int_0^T \widetilde{\theta}_s dW_s.
\end{eqnarray}
Moreover, $\widetilde{W}+\int_0^T \widetilde{\theta}_s \theta_s ds$ is a standard normal variable under $\mathbb{Q}$. Finally, the conditional distribution of $\widetilde{W}+\int_0^T \widetilde{\theta}_s \theta_s ds$ given $v$, i.e., $\mathcal{D}(\widetilde{W}+\int_0^T \widetilde{\theta}_s \theta_s ds | v)$, is standard normal under $\mathbb{Q}$.
\end{lemma} 
 Our main result is the next theorem.
\begin{theorem}
\label{th_approximation_random_twoboundary}
Under \textbf{Assumption D},
we have
\begin{eqnarray}
\nonumber
\proba (\Tau^Y_{b,c} \leq T | W_T,v) & = & \exp{\big(-\alpha W_T + \frac{1}{2} \int_0^T \theta_s^2 ds\big)} \\& &\times \esp_{\mathbb{Q}} \big[\mathbf{1}_{\{\Tau^Y_{b,c} \leq T \}}  \exp{\big(-\widetilde{\alpha} \widetilde{W}\big)} | W_T,v \big] \label{eq_nonapproximation_random_twoboundary}.
\end{eqnarray}
\end{theorem}
 We first calculate $\mathbb{Q} (\Tau^Y_{b,c} \leq T | W_T,v)$.
\begin{lemma}
\label{prop_condproba_random_twoboundary1}
Under \textbf{Assumption D}, we have
\begin{eqnarray}
\label{propeq_condproba_random_twoboundary1}
\mathbb{Q} (\Tau^Y_{b,c} \leq T | W_T,v) & = &  \sum_{j=1}^{\infty} q^Y_{b,c}(j | Y_T) \mathbf{1}_{\{Y_T \in [c_T,b_T]\}} + \mathbf{1}_{\{Y_T \notin [c_T,b_T]\}}.
\end{eqnarray}
\end{lemma}
 The next theorem gives a formula based on the strong theoretical assumption (\ref{approximation_assumption_random_twoboundary}).
\begin{theorem}
\label{th_approximation2_random_twoboundary}
We assume that \textbf{Assumption D} and the following assumption 
\begin{eqnarray}
\label{approximation_assumption_random_twoboundary}
\esp_{\mathbb{Q}} \big[\mathbf{1}_{\{\Tau^Y_{b,c} \leq T \}}  \exp{\big(-\widetilde{\alpha} \widetilde{W}\big)} | W_T ,v\big] \\ \nonumber = \esp_{\mathbb{Q}} \big[\mathbf{1}_{\{\Tau^Y_{b,c} \leq T \}}  | W_T,v \big] \esp_{\mathbb{Q}} \big[  \exp{\big(-\widetilde{\alpha} \widetilde{W}\big)} | W_T,v \big] 
\end{eqnarray}
holds. Then, we have
\begin{eqnarray}
\label{eq_approximation_random_twoboundary}
\proba (\Tau^Y_{b,c} \leq T | W_T,v) & = & \exp{\big(-\alpha W_T + \frac{1}{2} \int_0^T \theta_s^2 ds\big)} \\ && \nonumber \times \Big(\sum_{j=1}^{\infty} q^Y_{b,c}(j | Y_T) \mathbf{1}_{\{Y_T \in [c_T,b_T]\}} \\ &&+ \mathbf{1}_{\{Y_T \notin [c_T,b_T]\}} \Big) \exp{\big(\widetilde{\alpha} \int_0^T \widetilde{\theta}_s \theta_s ds\big)} \mathcal{L}_N(\widetilde{\alpha}).
\end{eqnarray}
\end{theorem}
 Finally, we get $P_b^Y(T)$ in the next corollary, by integrating $\proba (\Tau^Y_b \leq T | W_T,v)$ with respect to the value of $(W_T,v)$. We define the arrival space and cdf of $v$ as respectively $\Pi_v$ and $P_v$. Moreover,  we define $y_u$, $y_b$, $y_\theta$, etc. following the above definitions when integrating with respect to $y \in \Pi_v$.
\begin{corollary}
\label{th_proba_random_twoboundary1}
Under \textbf{Assumption D}, we have
\begin{eqnarray}
\nonumber 
P_{b,c}^Y(T) & = & 1 - \phi (\frac{b_T-u_T}{\sqrt{T}}) + \phi (\frac{c_T-u_T}{\sqrt{T}}) \\&& \nonumber + \int_{c_T-u_T}^{b_T-u_T}\int_{\Pi_v} \frac{1}{\sqrt{2 \pi T}} \exp{\Big(-\frac{x^2}{2T}\Big)} \exp{\big(-y_\alpha x + \frac{1}{2} \int_0^T y_{\theta,s}^2 ds\big)} \\ \label{eq_proba_random_twoboundary} &&\times \esp_{\mathbb{Q}} \big[\mathbf{1}_{\{\Tau^Y_{b,c} \leq T \}}  \exp{\big(-\widetilde{\alpha} \widetilde{W}\big)} | W_T =x,v=y \big] dx dP_v(y).
\end{eqnarray}
If we further assume  (\ref{approximation_assumption_random_twoboundary}), we have
\begin{eqnarray}
\nonumber 
P_{b,c}^Y(T) & = & 1 - \phi (\frac{b_T-u_T}{\sqrt{T}}) + \phi (\frac{c_T-u_T}{\sqrt{T}}) \\&& \nonumber + \int_{c_T-u_T}^{b_T-u_T} \int_{\Pi_v} \frac{1}{\sqrt{2 \pi T}} \exp{\Big(-\frac{x^2}{2T}\Big)} \exp{\big(-y_\alpha x + \frac{1}{2} \int_0^T y_{\theta,s}^2 ds\big)} \\ & & \nonumber \times \Big(\sum_{j=1}^{\infty} y^{x+y_{u,T}}_{q,y_b,y_c}(j | x+y_{u,T}) \mathbf{1}_{\{x \in [y_{c,T}-y_{u,T},y_{b,T}-y_{u,T}]\}} \\ & & \nonumber + \mathbf{1}_{\{x \notin [y_{c,T}-y_{u,T},y_{b,T}-y_{u,T}]\}}\Big) \\ & & \label{eq_proba_approximation_random_twoboundary} \times \exp{\big(y_{\widetilde{\alpha}} \int_0^T y_{\widetilde{\theta},s} y_{\theta,s} ds\big)} \mathcal{L}_N(y_{\widetilde{\alpha}}) dx dP_v(y).
\end{eqnarray}
\end{corollary}

\section{Proofs in the two-sided stochastic boundary process case}
In this section, we consider the proofs in the case when the two-sided boundary and the drift are stochastic processes and the variance is random.

The elementary idea in the proofs of this section is to condition by both $W_T$ and $v$, i.e., to derive results of the form $\proba (\Tau^Y_{b,c} \leq T | W_T,v)$. The proof of Proposition \ref{th_condproba_random_twoboundary} is based on Lemma \ref{lemma_girsanov_random_twoboundary}.
\begin{proof}[Proof of Proposition \ref{th_condproba_random_twoboundary}] 
By definition of the conditional probability, Equation (\ref{eq_condproba_random_twoboundary0}) can be rewritten formally as 
\begin{eqnarray}
\label{eq_condexp_random_twoboundary0}
\esp_\proba \big[\mathbf{1}_{\{\Tau^Y_b \leq T \}} | W_T,v \big]& =& \esp_{\mathbb{Q}} \big[\mathbf{1}_{\{\Tau^Y_{b,c} \leq T \}} M_T^{-1} | W_T ,v\big].
\end{eqnarray}
For any $\sigma(W_T,v)$-measurable event $E_T$, we can use a change of probability in the expectation by Lemma \ref{lemma_girsanov_random_twoboundary}, along with \textbf{Assumption D}, and we obtain that
\begin{eqnarray}
\label{eq_condexp_random_twoboundary00}
\esp_\proba \big[\mathbf{1}_{\{\Tau^Y_{b,c} \leq T \}} \mathbf{1}_{E_T}\big]& =& \esp_{\mathbb{Q}} \big[\mathbf{1}_{\{\Tau^Y_{b,c} \leq T \}} M_T^{-1} \mathbf{1}_{E_T} \big].
\end{eqnarray}
We can deduce Equation (\ref{eq_condexp_random_twoboundary0}) from Equation (\ref{eq_condexp_random_twoboundary00}) by definition of the conditional expectation. By definition of the conditional probability, Equation (\ref{eq_condproba_random_twoboundary}) can be rewritten formally as 
\begin{eqnarray}
\label{eq_condexp_random_twoboundary}
\esp_\proba \big[\mathbf{1}_{\{\Tau^Y_{b,c} \leq T \}} | W_T ,v\big]\\ \nonumber = \esp_{\mathbb{Q}} \big[ M_T^{-1} \esp_\mathbb{Q} \big[\mathbf{1}_{\{\Tau^Y_{b,c} \leq T \}} | W_T,\overline{W}_T,v \big] | W_T,v \big].
\end{eqnarray}
By definition of the conditional expectation, we can deduce what follows. If we can show that for any $E_T$, which is $\sigma(W_T,v)$-measurable, that 
\begin{eqnarray}
\label{eq_condexpevent_random_twoboundary}
\esp_\proba \big[\mathbf{1}_{\{\Tau^Y_{b,c} \leq T \}} \mathbf{1}_{E_T} \big] \\ \nonumber = \esp_\proba \Big[ \esp_{\mathbb{Q}} \big[ M_T^{-1} \esp_\mathbb{Q} \big[\mathbf{1}_{\{\Tau^Y_{b,c} \leq T \}} | W_T,\overline{W}_T,v \big] | W_T,v \big] \mathbf{1}_{E_T} \Big],
\end{eqnarray}
then Equation (\ref{eq_condexp_random_twoboundary}) holds. Let $E_T$ be a $\sigma(W_T,v)$-measurable event. By Lemma \ref{lemma_girsanov_random_twoboundary} along with \textbf{Assumption D}, we  obtain that
\begin{eqnarray}
\label{proof0_random_twoboundary}
\esp_\proba \big[\mathbf{1}_{\{\Tau^Y_{b,c} \leq T \}} \mathbf{1}_{E_T} \big] & = & \esp_{\mathbb{Q}} \big[\mathbf{1}_{\{\Tau^Y_{b,c} \leq T \}} \mathbf{1}_{E_T} M_T^{-1}\big].
\end{eqnarray}
Then, we have by the law of total expectation that
\begin{eqnarray}
\label{proof1_random_twoboundary}
\esp_{\mathbb{Q}} \big[\mathbf{1}_{\{\Tau^Y_{b,c} \leq T \}} \mathbf{1}_{E_T} M_T^{-1}\big]   = \\ \nonumber \esp_{\mathbb{Q}} \big[ \esp_{\mathbb{Q}} \big[\mathbf{1}_{\{\Tau^Y_{b,c} \leq T \}} \mathbf{1}_{E_T} M_T^{-1} | W_T, \overline{W}_T ,v\big] \big].
\end{eqnarray}
Since $\mathbf{1}_{E_T}$ and $M_T^{-1}$ are $\sigma(W_T, \overline{W}_T ,v)$-measurable random variables, we can pull them out of the conditional expectation and deduce that
\begin{eqnarray}
\label{proof2_random_twoboundary}
\esp_{\mathbb{Q}} \big[ \esp_{\mathbb{Q}} \big[\mathbf{1}_{\{\Tau^Y_{b,c} \leq T \}} \mathbf{1}_{E_T} M_T^{-1} | W_T, \overline{W}_T,v\big] \big] \\ \nonumber =  \esp_{\mathbb{Q}} \big[ \mathbf{1}_{E_T} M_T^{-1}\esp_{\mathbb{Q}} \big[\mathbf{1}_{\{\Tau^Y_{b,c} \leq T \}} | W_T, \overline{W}_T,v \big] \big].
\end{eqnarray}
If we use Equations (\ref{proof0_random_twoboundary})-(\ref{proof1_random_twoboundary})-(\ref{proof2_random_twoboundary}), we can deduce that Equation (\ref{eq_condexpevent_random_twoboundary}) holds.
\end{proof}
 In what follows, we give the proof of Lemma \ref{lemma_representation_random_twoboundary}. 
\begin{proof}[Proof of Lemma \ref{lemma_representation_random_twoboundary}] 
By \textbf{Assumption D},  we can deduce that $ 0 < \int_0^T \theta_s^2 ds < \infty$ a.s.. Thus, we can normalize $\overline{W}_T$ by $\sqrt{\int_0^T \theta_s^2 ds}$ a.s. and we have that $\frac{\overline{W}_T}{\sqrt{\int_0^T \theta_s^2 ds}}$ is a mixed normal random variable a.s. by definition. Using the same arguments from the proof of Lemma 2.10, we have that its conditional mean under $\proba$ is a.s. equal to
\begin{eqnarray}
\label{proof_mean_random_twoboundary}
\esp_\proba \Big[\frac{\int_0^T \theta_s dW_s}{\sqrt{\int_0^T \theta_s^2 ds}}\Big| v \Big] \nonumber
 &=& 0.
\end{eqnarray}
We also have that its conditional variance under $\proba$ is a.s. equal to
\begin{eqnarray}
\label{proof_variance_random_twoboundary}
\operatorname{Var}_\proba \Big(\frac{\int_0^T \theta_s dW_s}{\sqrt{\int_0^T \theta_s^2 ds}}\Big| v\Big) &=& 1.
\end{eqnarray}
Since its conditional mean and conditional variance are nonrandom, we obtain that its mean under $\proba$ is equal to $\esp_\proba \Big[\frac{\int_0^T \theta_s dW_s}{\sqrt{\int_0^T \theta_s^2 ds}}\Big]
= \esp_\proba \Big[\esp_\proba \Big[\frac{\int_0^T \theta_s dW_s}{\sqrt{\int_0^T \theta_s^2 ds}}\Big| v \Big]\Big]= 0$ by the the law of total expectation and Equation (\ref{proof_mean_random_twoboundary}). Similarly, we obtain that its variance is equal to 1 by the law of total expectation and Equation (\ref{proof_variance_random_twoboundary}).
Thus, we have that $\frac{\overline{W}_T}{\sqrt{\int_0^T \theta_s^2 ds}}$ is a standard normal random variable under $\proba$. Since $(W_T,\frac{\overline{W}_T}{\sqrt{\int_0^T \theta_s^2 ds}})$ is a centered normal random vector under $\proba$, there exists a standard normal random variable $\widetilde{W}$ under $\proba$ which is independent of $W_T$ and such that Equation (\ref{def_Wtilde_random_twoboundary}) holds. Then, we can calculate that the covariance between $W_T$ and $\frac{\overline{W}_T}{\sqrt{\int_0^T \theta_s^2 ds}}$ under $\proba$ is equal to
\begin{eqnarray}
\label{proof_cov_random_twoboundary}
\operatorname{Cov}_\proba \Big(W_T,\frac{\overline{W}_T}{\sqrt{\int_0^T \theta_s^2 ds}} \Big) &=& \esp_\proba \Big[\frac{\int_0^T \theta_s ds }{\sqrt{\int_0^T \theta_s^2 ds}}\Big].
\end{eqnarray}
Now, we can calculate that the correlation between $W_T$ and $\frac{\overline{W}_T}{\sqrt{\int_0^T \theta_s^2 ds}}$ under $\proba$ is equal to
\begin{eqnarray*}
\rho & = & \frac{1}{T}\esp_\proba \Big[\frac{\int_0^T \theta_s ds }{\sqrt{\int_0^T \theta_s^2 ds}}\Big].
\end{eqnarray*} 
Equation (\ref{def_Wtilde0_random_twoboundary}) can be deduced directly from Equation (\ref{def_Wtilde_random_twoboundary}). Moreover, we can reexpress $\widetilde{W}$ as
\begin{eqnarray*}
\widetilde{W} & = & \int_0^T \widetilde{\theta}_s dW_s.
\end{eqnarray*}
Moreover, we can deduce that $\widetilde{W}+\int_0^T \widetilde{\theta}_s \theta_s ds$ is a standard normal variable under $\mathbb{Q}$. This is due to its expression (\ref{def_Wtilde1_random_twoboundary}) and since by Lemma \ref{lemma_girsanov_random_twoboundary} along with \textbf{Assumption D}, $Y$ is a Wiener process under $\mathbb{Q}$. Finally, $\mathcal{D}(\widetilde{W}+\int_0^T \widetilde{\theta}_s \theta_s ds | v)$ is standard normal under $\mathbb{Q}$ by Equation (\ref{def_Wtilde1_random_twoboundary}).
\end{proof}
 We provide now the proof of Theorem \ref{th_approximation_random_twoboundary}, which is based on Lemma \ref{lemma_representation_random_twoboundary}.
\begin{proof}[Proof of Theorem \ref{th_approximation_random_twoboundary}] Using the same arguments from the proof of Theorem 2.11, we can reexpress $M_T$ as
\begin{eqnarray}
\nonumber M_T & = &  \exp{\big(\alpha W_T - \frac{1}{2} \int_0^T \theta_s^2 ds\big)} \exp{\big(\widetilde{\alpha} \widetilde{W}\big)}.
\end{eqnarray}
Then, we have
\begin{eqnarray}
\nonumber\proba (\Tau^Y_{b,c} \leq T | W_T,v) & = & \exp{\big(-\alpha W_T + \frac{1}{2} \int_0^T \theta_s^2 ds\big)} \\ && \nonumber \times\esp_{\mathbb{Q}} \big[\mathbf{1}_{\{\Tau^Y_{b,c} \leq T \}}  \exp{\big(-\widetilde{\alpha} \widetilde{W}\big)} | W_T,v \big].
\end{eqnarray}
Thus, we have shown Equation (\ref{eq_nonapproximation_random_twoboundary}).
\end{proof}
 We now give the proof of Lemma \ref{prop_condproba_random_twoboundary1}.
\begin{proof}[Proof of Lemma \ref{prop_condproba_random_twoboundary1}] 
By definition of the conditional probability, Equation (\ref{propeq_condproba_random_twoboundary1}) can be rewritten formally as 
\begin{eqnarray}
\label{eq_condexp_random_twoboundary1}
\esp_\mathbb{Q} \big[\mathbf{1}_{\{\Tau^Y_{b,c} \leq T \}} | W_T \big]& =& \sum_{j=1}^{\infty} q^Y_{b,c}(j | Y_T) \mathbf{1}_{\{Y_T \in [c_T,b_T]\}} + \mathbf{1}_{\{Y_T \notin [c_T,b_T]\}}.
\end{eqnarray}
By Lemma \ref{lemma_girsanov_random_twoboundary} along with \textbf{Assumption D}, $Y$ is a Wiener process under $\mathbb{Q}$. Then, we have by Anderson (\citeyear{anderson1960modification}) (Theorem 4.2, pp. 178-179) that Equation (\ref{eq_condexp_random_twoboundary1}) holds.
\end{proof}
 We provide now the proof of Theorem \ref{th_approximation2_random_twoboundary}, which is based on Lemma \ref{prop_condproba_random_twoboundary1}.
\begin{proof}[Proof of Theorem \ref{th_approximation2_random_twoboundary}] We have
\begin{eqnarray}
\nonumber \esp_{\mathbb{Q}} \big[\mathbf{1}_{\{\Tau^Y_{b,c} \leq T \}}  \exp{\big(-\widetilde{\alpha} \widetilde{W}\big)} | W_T,v \big]  & = & \esp_{\mathbb{Q}} \big[\mathbf{1}_{\{\Tau^Y_{b,c} \leq T \}}  | W_T,v \big] \\ && \nonumber \times \esp_{\mathbb{Q}} \big[  \exp{\big(-\widetilde{\alpha} \widetilde{W}\big)} | W_T,v \big]\\ \nonumber & = & \Big(\sum_{j=1}^{\infty} q^Y_{b,c}(j | Y_T) \mathbf{1}_{\{Y_T \in [c_T,b_T]\}} \\ && \label{proof_approximation_random_twoboundary} + \mathbf{1}_{\{Y_T \notin [c_T,b_T]\}}\Big)\\ && \nonumber \times \esp_{\mathbb{Q}} \big[  \exp{\big(-\widetilde{\alpha} \widetilde{W}\big)} | W_T,v \big],
\end{eqnarray}
where we use Assumption (\ref{approximation_assumption_random_twoboundary}) in the first equality, and Equation (\ref{propeq_condproba_random_twoboundary1}) from Lemma \ref{prop_condproba_random_twoboundary1} along with \textbf{Assumption D} in the second equality. Finally, we have
\begin{eqnarray}
\nonumber \esp_{\mathbb{Q}} \big[  \exp{\big(-\widetilde{\alpha} \widetilde{W}\big)} | W_T,v \big] & = & \esp_{\mathbb{Q}} \big[  \exp{\big(-\widetilde{\alpha} \widetilde{W}\big)} | v\big]\\ \nonumber & = & \exp{\big(\widetilde{\alpha} \int_0^T \widetilde{\theta}_s \theta_s ds\big)} \\ && \nonumber \times\esp_{\mathbb{Q}} \big[  \exp{\big(-\widetilde{\alpha} \big(\widetilde{W}+\int_0^T \widetilde{\theta}_s \theta_s ds\big)\big)} | v \big]\\ \nonumber & = & \exp{\big(\widetilde{\alpha} \int_0^T \widetilde{\theta}_s \theta_s ds\big)} \\ && \nonumber \times\esp_{\mathbb{Q}} \big[  \exp{\big(-\widetilde{\alpha} \big(\widetilde{W}+\int_0^T \widetilde{\theta}_s \theta_s ds\big)\big)} \big]\\ \nonumber & = & \exp{\big(\widetilde{\alpha} \int_0^T \widetilde{\theta}_s \theta_s ds\big)} \esp_{\proba} \big[  \exp{\big(-\widetilde{\alpha} N \big)} \big]\\ \label{proof_approximation0_random_twoboundary} & = & \exp{\big(\widetilde{\alpha} \int_0^T \widetilde{\theta}_s \theta_s ds\big)} \mathcal{L}_N(\widetilde{\alpha}).
\end{eqnarray}
Here, we use the fact that $\widetilde{W}$ is independent from $W_T$ in the first equality, the fact that $\theta_t$ and $\widetilde{\theta}_t$ for any $t \in [0,T]$ are $\sigma(v)$-measurable random variables in the second equality, the fact that $\mathcal{D}(\widetilde{W}+\int_0^T \widetilde{\theta}_s \theta_s ds | v)$ is standard normal under $\mathbb{Q}$ by Lemma \ref{lemma_representation_random_twoboundary} along with \textbf{Assumption D} in the third equality, the fact that $\widetilde{W}+\int_0^T \widetilde{\theta}_s \theta_s ds$ is a standard normal variable under $\mathbb{Q}$ by Lemma \ref{lemma_representation_random_twoboundary} along with \textbf{Assumption D} in the fourth equality, and Equation (2.14
) in the last equality. We can deduce Equation (\ref{eq_approximation_random_twoboundary}) from Equations (\ref{eq_nonapproximation_random_twoboundary}), (\ref{proof_approximation_random_twoboundary}) and (\ref{proof_approximation0_random_twoboundary}).
\end{proof}
 Finally, we get $P_{b,c}^Y(T)$ in the next theorem, by integrating $\proba (\Tau^Y_{b,c} \leq T | W_T,v)$ with respect to the value of $(W_T,v)$.
\begin{proof}[Proof of Corollary \ref{th_proba_random_twoboundary1}] 
We can calculate that
\begin{eqnarray*}
P_{b,c}^Y(T) & = & \int_{-\infty}^{\infty} \int_{\Pi_v} \proba (\Tau^Y_{b,c} \leq T | W_T = x,v=y) \\ && \nonumber \times\frac{1}{\sqrt{2 \pi T}} \exp{\Big(-\frac{x^2}{2T}\Big)} dx dP_v(y)\\ & = & 1 - \phi (\frac{b_T-u_T}{\sqrt{T}}) + \phi (\frac{c_T-u_T}{\sqrt{T}}) \\&&+ \int_{c_T-u_T}^{b_T-u_T} \int_{\Pi_v} \proba (\Tau^Y_b \leq T | W_T = x,v=y) \\ & & \nonumber \times \frac{1}{\sqrt{2 \pi T}} \exp{\Big(-\frac{x^2}{2T}\Big)} dx dP_v(y)\\ & = & 1 - \phi (\frac{b-u_T}{\sqrt{T}}) \\ & & \nonumber+ \int_{-\infty}^{b-u_T} \int_{\Pi_v} \frac{1}{\sqrt{2 \pi T}} \exp{\Big(-\frac{x^2}{2T}\Big)} \exp{\big(-y_\alpha x + \frac{1}{2} \int_0^T y_{\theta,s}^2 ds\big)} \\ &&\times \esp_{\mathbb{Q}} \big[\mathbf{1}_{\{\Tau^Y_b \leq T \}}  \exp{\big(-\widetilde{\alpha} \widetilde{W}\big)} | W_T =x,v=y \big] dx dP_v(y).
\end{eqnarray*}
Here, we use Equation (\ref{PgZdef_random_twoboundary}), regular conditional probability and the fact that $W_T$ and $v$ are independent in the first equality, the fact that $\proba (\Tau^Y_{b,c} \leq T | W_T = x)=1$ for any $x \geq b_T-u_T$ and any $x \leq c_T-u_T$ in the second equality, and Equation (\ref{eq_nonapproximation_random_twoboundary}) in the third equality. We have thus shown Equation (\ref{eq_proba_random_twoboundary}). Equation (\ref{eq_proba_approximation_random_twoboundary}) can be shown following the same first two equalities and using Equation (\ref{eq_approximation_random_twoboundary}) in the third equality.
\end{proof}

\end{appendix}


\bibliographystyle{imsart-nameyear} 
\bibliography{biblio}       

\end{document}